\documentclass[a4paper,11pt,reqno]{amsart}

\usepackage{fullpage}

\usepackage{dsfont,amssymb, mathtools,mathrsfs,scalerel,mathalfa}

\usepackage{color}

\usepackage{hyperref} 

\numberwithin{equation}{section}

\newtheorem{theorem}{Theorem}[section]
\newtheorem{corollary}[theorem]{Corollary}
\newtheorem{assumption}[theorem]{Assumption}
\newtheorem{lemma}[theorem]{Lemma}
\newtheorem{proposition}[theorem]{Proposition}
\theoremstyle{remark}
\newtheorem{remark}[theorem]{Remark}
\newtheorem{example}[theorem]{Example}
\theoremstyle{definition}
\newtheorem{definition}[theorem]{Definition}

\newcommand\bp{\begin{proof}}
\newcommand\ep{\end{proof}}

\newcommand\bicrossl{%
  \mathrel{\scalerel*{\mathrel{\triangleright}\joinrel\blacktriangleleft}{x}}}
\newcommand\bicrossr{%
  \mathrel{\scalerel*{\blacktriangleright\joinrel\mathrel{\triangleleft}}{x}}}
\newcommand\biopencrossl{%
  \mathrel{\scalerel*{>\kern-.4\LMpt\joinrel\blacktriangleleft}{x}}}
\newcommand\biopencrossr{%
  \mathrel{\scalerel*{\blacktriangleright\joinrel\kern-.4\LMpt<}{x}}}

\newcommand\Dhat{{\hat\Delta}}

\newcommand\Ad{\operatorname{Ad}}
\newcommand\Aff{\operatorname{Aff}}
\newcommand\Aut{\operatorname{Aut}}
\newcommand\Dom{\operatorname{Dom}}
\newcommand\End{\operatorname{End}}

\newcommand\HS{\operatorname{HS}}
\newcommand\id{\operatorname{id}}

\newcommand\Op{\operatorname{Op}}
\newcommand\OKN{\operatorname{Op}_{\mathrm{KN}}}
\newcommand\Tr{\operatorname{Tr}}

\newcommand{\C}{{\mathbb C}}
\newcommand{\R}{{\mathbb R}}
\newcommand\T{{\mathbb T}}
\newcommand\Z{{\mathbb Z}}

\newcommand\g{{\mathfrak g}}

\newcommand\un{\mathds{1}}

\newcommand{\F}{{\mathcal F}}
\newcommand{\G}{{\mathcal G}}

\newcommand{\N}{{\mathcal N}}

\newcommand{\U}{{\mathcal U}}

\newcommand{\CE }{\mathcal E}

\newcommand{\CL}{\mathcal L}

\newcommand{\CJ}{\mathcal J}

\newcommand{\CB}{\mathcal B}
\newcommand{\CU}{\mathcal U}
\newcommand{\CF}{\mathcal F}

\newcommand{\eps}{\varepsilon}

\newcommand{\vf}{\varphi}

\begin{document}

\title{Quantization of subgroups of the affine group}

\date{June 5, 2019; minor changes November 21, 2020}

\author[Bieliavsky]{P. Bieliavsky}

\email{Pierre.Bieliavsky@uclouvain.be}

\address{Institut de Recherche en Math\'ematique et Physique, Universit\'e Catholique de Louvain, Chemin du Cyclotron, 2, 1348 Louvain-la-Neuve, Belgium}

\author[Gayral]{V. Gayral}

\email{victor.gayral@univ-reims.fr}

\address{Laboratoire de Math\'ematiques, CNRS UMR 9008, Universit\'e de Reims Champagne-Ardenne,
Moulin de la Housse - BP 1039,
51687 Reims, France}

\author[Neshveyev]{S. Neshveyev}

\email{sergeyn@math.uio.no}

\address{Department of Mathematics, University of Oslo,
P.O. Box 1053 Blindern, NO-0316 Oslo, Norway}

\thanks{}

\author[Tuset]{L. Tuset}

\email{larst@oslomet.no}

\address{Department of Computer Science, OsloMet - storbyuniversitetet,
P.O. Box 4 St. Olavs plass, NO-0130 Oslo, Norway}

\begin{abstract}
Consider a locally compact group $G=Q\ltimes V$ such that $V$ is abelian and the action of $Q$ on the dual abelian group $\hat V$ has a free orbit of full measure. We show that such a group~$G$ can be quantized in three equivalent ways:
\begin{enumerate}
\item by reflecting across the Galois object defined by the canonical irreducible representation of $G$ on $L^2(V)$;
\item by twisting the coproduct on the group von Neumann algebra of $G$ by a dual $2$-cocycle obtained from the
$G$-equivariant Kohn--Nirenberg quantization of $V\times\hat V$;
\item by considering the bicrossed product defined by a matched pair of subgroups of $Q\ltimes\hat V$ both isomorphic to $Q$.
\end{enumerate}

In the simplest case of the $ax+b$ group over the reals, the dual cocycle in (2) is an analytic analogue of the Jordanian twist. It was first found by Stachura using different ideas. The equivalence of approaches~(2) and~(3) in this case implies that
the quantum $ax+b$ group of Baaj--Skandalis is isomorphic to the quantum group defined by Stachura.

Along the way we prove a number of results for arbitrary locally compact groups $G$. Using recent results of De Commer we show that a class of $G$-Galois objects is parametrized by certain cohomology classes in $H^2(G;\T)$. This extends results of Wassermann and Davydov in the finite group case. A new phenomenon is that already the unit class in $H^2(G;\T)$ can correspond to a nontrivial Galois object. Specifically, we show that any nontrivial locally compact group~$G$ with group von Neumann algebra a factor of type I admits a canonical cohomology class of dual $2$-cocycles such that the corresponding quantization of $G$ is neither commutative nor cocommutative.
\end{abstract}

\maketitle
\tableofcontents

\section*{Introduction}

Although the problem of quantization of Lie bialgebras was solved in full generality more than 20 years ago by Etingof and Kazhdan~\cite{EK}, the list of noncompact Poisson--Lie groups admitting nonformal (analytic) quantizations is still quite short. The difficulty lies not only in making sense of certain formal constructions, but in that there exist real obstacles in doing~so. A famous example is the group $SU(1,1)$. At the Hopf $*$-algebraic level its quantization is well understood, but by a ``no go" result of Woronowicz there is no way of making sense of it at the operator algebraic level~\cite{Wor}. As was first realized by Korogodsky~\cite{Kor} and then completed by Koelink and Kustermans~\cite{KK}, the right group to quantize in this case is the nonconnected group $SU(1,1)\rtimes\Z/2\Z$, the normalizer of $SU(1,1)$ in $SL(2,\C)$.

The present paper is motivated by an even easier example, the $ax+b$ group $G$ over the reals. Its Lie algebra $\mathfrak g$ is generated by two elements $x,y$ such that $[x,y]=y$. Consider the Lie bialgebra $(\g,\delta)$, with the cobracket $\delta$ defined by the triangular $r$-matrix
$$
r:=x\otimes  y-y\otimes x.
$$
It can be explicitly quantized using the \emph{Jordanian twist}
$$
\Omega:=\exp\{x\otimes\log(1+h y)\}\in(U\g\otimes U\g)[[h]]
$$
found independently by Coll--Gerstenhaber--Giaquinto~\cite{CGG} and Ogievetsky~\cite{Og}.\footnote{In \cite{CGG}, the Jordanian twist does not appear in exponential form.
To our knowledge, it was first observed in~\cite{GV} that Coll--Gerstenhaber--Giaquinto's twist can be put in the exponential form and therefore it coincides with Ogievetsky's twist.} This twist and its generalizations have been extensively studied, see, e.g.,~\cite{KLM},~\cite{KLS}. It is particularly popular in physics literature, as it can be used to construct the $\kappa$-Minkowski space~\cite{BP}, which by \cite{MR} can also be obtained from a bicrossed product construction.

If we want to make sense of $\Omega$  as a unitary operator on $L^2(G\times G)$, since the elements~$x$ and~$y$ are skew-adjoint, our best bet is to take $h\in i\R$, but then we still have a problem with the logarithm, as the spectrum of $y$ is the entire line $i\R$. A correct analytic analogue of the Jordanian twist was found by Stachura~\cite{Stachura}, see formula~\eqref{eq:Stachura} below, but it turns out that, similarly to the case of $SU(1,1)$, it is important to work with the entire nonconnected $ax+b$ group. What to do in the connected case remains an open problem.

In fact, in an earlier paper~\cite{BG} the first two authors found a \emph{universal deformation formula} for the actions of the connected component of the $ax+b$ group (and, more generally, of K\"ahlerian Lie groups) on C$^*$-algebras. Unfortunately, despite the claims in~\cite{NT} and~\cite{BGNT1}, this formula turned out to define a coisometric but nonunitary \emph{dual $2$-cocycle}, which is the term we prefer to use in the analytical setting instead of the ``twist". See Remark~\ref{rem:ax+b} below and erratum to~\cite{BGNT1} for further discussion.

If we do consider the nonconnected $ax+b$ group, then even earlier, Baaj and Skandalis constructed its quantization as a bicrossed product of two copies of $\R^*$~\cite{BS2}. One disadvantage of this construction is that from the outset it is not clear how justifiable it is to call their quantum group a quantization of the $ax+b$ group. Some justification was given later by Vaes and Vainerman~\cite{VV}.

The present work grew out of the natural question how the constructions in \cite{BS2}, \cite{BG} and \cite{Stachura} are related. As we already said, we found out that~\cite{BG} does not lead to a unitary cocycle and therefore cannot actually be used to quantize the $ax+b$ group. But the constructions in~\cite{BS2} and~\cite{Stachura} turned out to be equivalent, as was conjectured by Stachura. Furthermore, we found an interpretation of the Jordanian twist/Stachura cocycle in terms of the Kohn--Nirenberg quantization, which allowed us to construct quantum analogues of a class of semidirect products $Q\ltimes V$. We also realized that these constructions have a very natural description within De Commer's analytic version of the Hopf--Galois theory~\cite{DC,DC2}.

\medskip

In more detail, the main results and organization of the paper are as follows. After a short preliminary section, we begin by discussing $G$-Galois objects for general locally compact groups~$G$ in Section~\ref{sec:Galois}. These are von Neumann algebras equipped with actions of $G$ that are in an appropriate sense free and transitive. For compact groups such actions are known  in the operator algebra literature as full multiplicity ergodic actions. Using recent results of De Commer~\cite{DC2} we show that the $G$-Galois objects with underlying algebras factors of type I are classified by certain second cohomology classes on $G$ (Theorem~\ref{thm:I-Galois}). For finite groups such a result quickly leads to a complete classification of $G$-Galois objects obtained by Wassermann~\cite{Wa} (although the result is not very explicit there, see~\cite{NTcoc}) and Davydov~\cite{Da}. For infinite groups the situation is of course more complicated, as in general there exist Galois objects built on non-type-I algebras.

We show next that under extra assumptions a $G$-Galois object of the form $(B(H),\Ad\pi)$, where $\pi$ is a projective representation of $G$ on $H$, defines a dual unitary $2$-cocycle (Proposition~\ref{prop:dual-cocycle}). We do not know whether these assumptions are always automatically satisfied, but we show that they are if $\pi$ is a genuine representation (Theorem~\ref{thm:genuine-rep}). This implies that if the group von Neumann algebra $W^*(G)$ of a nontrivial group~$G$ is a type I factor, then there exists a canonical nontrivial cohomology class of dual cocycles on~$G$. This gives probably the shortest explanation why a quantization of, for example, the $ax+b$ group exists at the operator algebraic level. At the Lie (bi)algebra level this is related to Drinfeld's result on quantization of Frobenius Lie algebras~\cite{Dr}.

\smallskip

The construction of the dual cocycle in Section~\ref{sec:Galois} is, however, rather inexplicit and in Section~\ref{sec:KN} we find a formula for such a cocycle for the semidirect products $G=Q\ltimes V$ such that~$V$ is abelian and the action of $Q$ on the dual abelian group $\hat V$ has a free orbit of full measure (Assumption~\ref{Fro}). It is well-known that producing a dual $2$-cocycle/twist is essentially equivalent to finding a $G$-equivariant deformation of an appropriate algebra of functions on~$G$. Our assumptions on $G=Q\ltimes V$ imply that we can identify $L^2(G)$ with $L^2(V\times\hat V)$ in a $G$-equivariant way. The Kohh--Nirenberg quantization of $V\times\hat V$ provides then a deformation of $L^2(G)$ and gives rise to a dual unitary $2$-cocycle $\Omega$ (Theorem~\ref{thm:KN1}). The cohomology class of this cocycle is exactly the one we defined in Section~\ref{sec:Galois} (Theorem~\ref{thm:KN2}).

In fact, there are two versions of the Kohn--Nirenberg quantization, so we get two cohomologous dual cocycles. In the case of the $ax+b$ group we show that one of these cocycles coincides with Stachura's cocycle (Proposition~\ref{prop:Stachura}).

\smallskip

Finally, in Section~\ref{sec:bicross} we consider the bicrossed product defined by a matched pair of two copies of $Q$ in $Q\ltimes\hat V$. We show that this quantum group is self-dual and isomorphic to $(W^*(G),\Omega\Dhat(\cdot)\Omega^*)$ (Theorem~\ref{thm:bicrossed} and Corollary~\ref{cor:bicrossed}). This is achieved by showing that the multiplicative unitary of the twisted quantum group $(W^*(G),\Omega\Dhat(\cdot)\Omega^*)$ is given by a \emph{pentagonal transformation} on $Q\times\hat V$ (Theorem~\ref{thm:mult-unitary}) and by applying the Baaj--Skandalis procedure of reconstructing a matched pair of groups from such a transformation~\cite{BS3}.

\section{Preliminaries}

Let $G$ be a locally compact group. We fix a left invariant Haar measure $dg$ on $G$ and denote by $L^p(G)$, $p\in[1,\infty]$, the associated function spaces.

The modular function $\Delta=\Delta_G$ is defined by the relation
$$
\int_G f(hg)\,dh=\Delta(g)^{-1}\int_G f(h)\,dh\ \ \text{for}\ \ f\in C_c(G).
$$
Then $\Delta(g)^{-1}dg$ is a right invariant Haar measure on $G$.

In a similar way, if $q\in\Aut(G)$, then the modulus $|q|=|q|_G$ of $q$ is defined by the identity
$$
\int_Gf(q(h))\,dh=|q|^{-1}\int_Gf(h)\,dh\ \ \text{for}\ \ f\in C_c(G).
$$

We let $\lambda$ and $\rho$  be the left and right regular (unitary) representations of $G$ on $L^2(G)$:
\begin{align*}
(\lambda_gf)(h)=f(g^{-1}h)\quad\mbox{and}\quad(\rho_gf)(h)=\Delta(g)^{1/2}f(hg).
\end{align*}

For a function $f$ on $G$ we define a function $\check f$ by
$$
\check f(g):=f(g^{-1}).
$$
We also let $J=J_G$ and $\hat J=\hat J_G$ be the modular conjugations of $L^\infty(G)$ and $W^*(G):=\lambda(G)''$:
$$
Jf:=\bar f\quad\mbox{and}\quad \hat Jf:=\Delta^{-1/2}\bar{\check f},
$$
and we use the shorthand notation
\begin{align}
\label{J}
\mathcal J:=J\hat J=\hat JJ,\ \ \text{so that}\ \ \CJ f=\Delta^{-1/2}\check f.
\end{align}

The multiplicative unitary $W=W_G\colon L^2(G)\otimes L^2(G)\to L^2(G)\otimes L^2(G)$ of $G$ is defined by
$$
(Wf)(g,h)=f(g,g^{-1}h).
$$
The multiplicative unitary $\hat W=\hat W_G$ of the dual quantum group is defined by
$$
\hat W=W^*_{21},\ \ \text{so that}\ \ (\hat Wf)(g,h)=f(hg,h).
$$

The  coproduct $\Dhat\colon W^*(G)\to W^*(G)\bar\otimes W^*(G)$ on the group von Neumann algebra $W^*(G)$ is defined  by $\Dhat(\lambda_g)=\lambda_g\otimes\lambda_g$. We then have
$$
\Dhat(x)=\hat W^*(1\otimes x)\hat W\ \ \text{for}\ \ x\in W^*(G).
$$

Let now  $V$ be a locally compact Abelian group and $\hat V$ be its Pontryagin dual.
Elements of~$V$ will be denoted by the Latin letters $v,v_j,v'\dots$ while elements of $\hat V$ will be denoted by
the Greek letters $\xi,\xi_j,\xi'\dots$. We will use additive notation both on $V$ and  on $\hat V$.

The duality paring $\hat V\times V\to \T$ will be denoted by $e^{i\langle\xi,v\rangle}$. This is just a notation, we do not claim that
there is an exponential function here. We also let
$$
e^{-i\langle\xi,v\rangle}:=\overline{e^{i\langle\xi,v\rangle}}=e^{i\langle-\xi,v\rangle}=e^{i\langle\xi,-v\rangle}.
$$

We fix  a Haar  measure $dv$ on $V$ and we normalize the Haar measure $d\xi$ of $\hat V$ so that the Fourier
transform $\CF_V$ defined by
$$
(\CF_Vf)(\xi):=\int_V e^{-i\langle\xi,v\rangle} f(v)\,dv
$$
becomes unitary from $L^2(V)$ to $L^2(\hat V)$. For functions in several variables only one of which is in $V$, we use the same symbol $\CF_V$ to denote the partial Fourier transform in that variable.

\section{Galois objects and dual cocycles}\label{sec:Galois}

\subsection{Projective representations and Galois objects}\label{ss:HG}

Hopf--Galois objects is a well-studied topic in Hopf algebra theory. An adaption of this notion to locally compact quantum groups has been developed by De Commer~\cite{DC}. Let us recall the main definitions. We will do this for genuine groups, as this is mainly the case we are interested in, but it will be important for us that the theory is developed at least for locally compact groups and their duals.

\smallskip

Let $G$ be a locally compact group and $\beta$ be an action of $G$ on a von Neumann algebra $\N$. Such an action is called integrable if the operator-valued weight
$$
P\colon \N\to \N^\beta,\ \ \N_+\ni a\mapsto\int_G\beta_g(a)\,dg,
$$
is semifinite. If $\beta$ is in addition ergodic, then we get a normal semifinite faithful weight $\tilde\varphi$ on~$\N$ such that $P(a)=\tilde\varphi(a)1$. Note that
\begin{equation}\label{eq:scaling}
\tilde\varphi(\beta_g(a))=\Delta(g)^{-1}\tilde\varphi(a)\ \ \text{for all}\ \ a\in\N_+.
\end{equation}
We can then define an isometric map
$$
{\mathcal G}\colon L^2(\N,\tilde\varphi)\otimes L^2(\N,\tilde\varphi)\to L^2(G;L^2(\N,\tilde\varphi)),\ \ {\mathcal G}\big(\tilde\Lambda(a)\otimes \tilde\Lambda(b)\big)(g)=\tilde\Lambda(\beta_g(a)b),
$$
where $\tilde\Lambda\colon{\mathfrak N}_{\tilde\varphi}\to L^2(\N,\tilde\varphi)$ denotes the GNS-map.
The pair $(\N,\beta)$ consisting of a von Neumann algebra $\N$ and an ergodic integrable action $\beta$ of~$G$ on~$\N$ is called a {\em $G$-Galois object} if the {\em Galois map} $\mathcal G$ is unitary.

The following characterization of Galois objects is often easier to use.

\begin{proposition}\label{prop:Galois-criterion}
A pair $(\N,\beta)$ consisting of a von Neumann algebra $\N$ and an ergodic integrable action $\beta$ of a locally compact group $G$ on $\N$ is a $G$-Galois object if and only if $\N\rtimes G$ is a factor, which is then necessarily of type I.
\end{proposition}

\bp
This is a simple consequence of results in~\cite[Section~2]{DC}. Indeed, the integrability assumption implies that $\N\rtimes G$
has a canonical representation $\eta$ on
$L^2(\N,\tilde\vf)$, and then by~\cite[Theorem~2.1]{DC} the pair $(\N,\beta)$ is a $G$-Galois object if and only if $\eta$
is faithful. If $\N\rtimes G$ is a factor, then $\eta$ is faithful, so we get one implication in the proposition.

Next, the ergodicity of the action $\beta$ implies that the action on $\N'$ by the automorphisms $\Ad\eta(\lambda_g)$ is ergodic as well, that is, $\N'\cap\eta(W^*(G))'=\C1$. It follows that $\eta(\N\rtimes G)=B(L^2(\N,\tilde\vf))$. Hence, if $(\N,\beta)$ is a Galois object, then $\N\rtimes G$ is a factor canonically isomorphic to  $B(L^2(\N,\tilde\vf))$.
\ep

We will be interested in $G$-Galois objects that are themselves factors of type I, in which case we say, following~\cite{DC2}, that $(\N,\beta)$ is a {\em I-factorial $G$-Galois object}. Identifying $\N$ with $B(H)$ for a Hilbert space $H$, we then get a projective unitary representation $\pi\colon G\to PU(H)$ such that $\beta_g=\Ad \pi(g)$. Note that the equivalence class of $\pi$ is uniquely determined by $(\N,\beta)$.

\begin{remark}
Ergodicity of $\beta=\Ad\pi$ is equivalent to irreducibility of $\pi$. Assuming ergodicity, integrability of the action $\Ad\pi$ is equivalent to square-integrability of the irreducible projective representation $\pi$, meaning that there are nonzero vectors $\xi,\zeta$ for which the function $g\mapsto(\pi(g)\xi,\zeta)$ is square-integrable. This observation goes back to~\cite[Example~2.8, Chapter~III]{CT}, but let us give some details.

Assume first that the action $\Ad\pi$ is integrable. Then the domain of definition of the weight~$\tilde\varphi$ must contain a nonzero rank-one operator $\theta_{\xi,\xi}$. Then, for every $\zeta\in H$, the function $g\mapsto((\Ad\pi(g))(\theta_{\xi,\xi})\zeta,\zeta)=|(\pi(g)\xi,\zeta)|^2$ is integrable, so $\pi$ is square-integrable.

Conversely, assume $\pi$ is square-integrable. Then by~\cite{DM} (for genuine representations) and by~\cite{A} (for projective representations) there exists a unique positive, possibly unbounded, nonsingular operator $K$ on $H$, called the \emph{Duflo--Moore formal degree operator}, such that
$$
\int_G|(\pi(g)\xi,\zeta)|^2dg=\|K^{1/2}\xi\|^2\|\zeta\|^2\ \ \text{for all}\ \ \xi\in\Dom(K^{1/2})\ \ \text{and}\ \ \zeta\in H.
$$
This implies that $\theta_{\xi,\xi}$ is in the domain of definition of the weight $\tilde\varphi$ and $\tilde\varphi(\theta_{\xi,\xi})=\|K^{1/2}\xi\|^2$. It follows that the action $\Ad\pi$ is integrable and
$$
\tilde\varphi=\Tr(K^{1/2}\cdot K^{1/2}).
$$

Note for future use that property~\eqref{eq:scaling} translates into
\begin{equation}\label{eq:scaling2}
(\Ad\pi(g))(K)=\Delta(g)K.
\end{equation}
Since the action $\Ad\pi$ is ergodic, this determines $K$ uniquely up to a scalar factor. In particular, as was already observed in~\cite{DM}, if $G$ is unimodular then $K$ is scalar, and otherwise $K$ is unbounded.
\end{remark}

\begin{remark}\label{rem:genuine-q}
By~\cite{DC}, given a $G$-Galois object, we get a locally compact quantum group~$G'$ obtained by \emph{reflecting $G$ across the Galois object}. If $G$ is abelian, then $G'=G$. But if $G$ is a nonabelian  genuine locally compact group and our Galois object has the form $(B(H),\Ad\pi)$ for a projective representation $\pi$ of $G$, then $G'$ is a genuine quantum group.

Indeed, assume $G'$ is a group. By the general theory we know that $B(H)$ is a $G'$-Galois object with respect to an action $\beta'$ of~$G'$ commuting with the action of $G$, see~\cite{DC}. There exist scalars $\chi_g(g')\in\T$ such that
$$
\beta'_{g'}(\tilde\pi(g))=\chi_g(g')\tilde\pi(g) \ \text{for all}\ \ g\in G,\ g'\in G',
$$
where $\tilde\pi(g)$ is any lift of $\pi(g)$ to $U(H)$.
Then $\chi_g$ is a character of $G'$. Furthermore, by ergodicity of the action $\beta'$, if $\chi_{g_1}=\chi_{g_2}$ for some $g_1,g_2\in G$, then $\tilde\pi(g_1)$ and $\tilde\pi(g_2)$ coincide up to a scalar factor, which by surjectivity of the Galois map for $(B(H),\Ad\pi)$ is possible only when $g_1=g_2$. Therefore the map $g\mapsto\chi_g$ is an injective homomorphism from $G$ into the group of characters of $G'$. Hence $G$ is abelian, which contradicts our assumption.
\end{remark}

By a recent duality result of De Commer~\cite{DC2}, for any locally compact quantum group $G$, there is a bijection between the isomorphism classes of I-factorial $G$-Galois object and I-factorial $\hat G$-Galois objects. This bijection is constructed as follows. Suppose we are given a I-factorial $G$-Galois object $(\N,\beta)$. Then $\N'\cap (\N\rtimes G)$, equipped with the dual action, becomes a I-factorial Galois object for $\hat G$. (More precisely, we rather get a Galois object for the opposite comultiplication on $L^\infty(\hat G)$ and then an additional application of modular conjugations is needed to really get a Galois object for $\hat G$, but this is unnecessary in our setting of genuine groups and their duals.)

There is a simple class of $\hat G$-Galois objects constructed as follows. Assume from now on that~$G$ is second countable. Let $\omega$ be a $\T$-valued Borel $2$-cocycle on $G$. Consider the $\omega$-twisted left regular representation of $G$ on $\lambda^\omega\colon G\to B(L^2(G))$ defined by
$$
(\lambda^\omega_gf)(h)=\omega(g,g^{-1}h)f(g^{-1}h),
$$
satisfying $\lambda^\omega_g\lambda^\omega_h=\omega(g,h)\lambda^\omega_{gh}$, and let $W^*(G;\omega):=\lambda^\omega(G)''\subset\CB(L^2(G))$. Then $W^*(G;\omega)$ equipped with the coaction $\lambda^{\omega}_g\mapsto \lambda^{\omega}_g\otimes\lambda_g$ of $G$ (or in other words, the action of $\hat G$) is a $\hat G$-Galois object, see~\cite[Section~5]{DC} for this statement in the setting of locally compact quantum groups.

In fact, this covers all possible I-factorial $\hat G$-Galois objects and by duality we get a description of the I-factorial $G$-Galois objects:

\begin{theorem}\label{thm:I-Galois}
For any second countable locally compact group $G$, there is a bijection between the isomorphism classes of I-factorial $G$-Galois objects and the cohomology classes $[\omega]\in H^2(G;\T)$ such that the twisted group von Neumann algebra $W^*(G;\omega)$ is a type I factor.
\end{theorem}

Here $H^2(G;\T)$ denotes the Moore cohomology of $G$, which is based on Borel cochains~\cite{Moore}.

\smallskip

Explicitly, the bijection in the theorem is defined as follows. To a I-factorial $G$-Galois object $(B(H),\Ad\pi)$ we associate the cohomology class $[\omega_\pi]\in H^2(G;\T)$ defined by the projective representation $\pi$ of $G$. Recall that this means that we lift $\pi$ to a Borel map $\tilde\pi\colon G\to U(H)$ and define a Borel $\T$-valued $2$-cocycle $\omega_\pi$ on $G$ by the identity
$$
\tilde\pi(g)\tilde\pi(h)=\omega_\pi(g,h)\tilde\pi(gh).
$$
The inverse map  associates to $[\omega]\in H^2(G;\T)$ (such that $W^*(G;\omega)$ is a type I factor) the isomorphism class of the $G$-Galois
object $(W^*(G;\omega),\Ad\lambda^\omega)$.

\smallskip

For finite groups $G$, this theorem is essentially due to Wassermann~\cite{Wa} (see also~\cite{NTcoc}), as well as to Davydov~\cite{Da} in the purely algebraic setting.

\smallskip

We divide the proof of Theorem~\ref{thm:I-Galois} into a couple of lemmas. Let $(B(H),\Ad\pi)$ be a I-factorial $G$-Galois object and $\omega$ be the cocycle defined by a lift $\tilde\pi$ of $\pi$.

\begin{lemma}
The $\hat G$-Galois object associated with the I-factorial $G$-Galois object $(B(H),\Ad\pi)$ is isomorphic to  $(W^*(G;\bar \omega),\alpha)$, where the coaction $\alpha$ of $G$
is defined by
\begin{equation}\label{eq:coaction}
\alpha(\lambda^{\bar \omega}_g)=\lambda^{\bar \omega}_g\otimes\lambda_g.
\end{equation}
\end{lemma}

\bp
We have an isomorphism
$$
B(H)\rtimes G\cong B(H)\bar\otimes W^*(G;\bar \omega),\ \ B(H)\ni T\mapsto T\otimes1,\ \ \lambda_g\mapsto \tilde\pi(g)\otimes\lambda^{\bar \omega}_g.
$$
Explicitly, the crossed product $B(H)\rtimes G$ is the von Neumann subalgebra of $B(L^2(G;H))$ generated by the operators $\lambda_g$ considered as operators on $L^2(G;H)$ and the operators $\tilde T$ for $T\in B(H)$ defined by
$$
(\tilde T\xi)(g)=\pi(g)^*T\pi(g)\xi(g).
$$
Then the required isomorphism is given by $\Ad U$, where $U\colon L^2(G;H)\to L^2(G;H)$ is defined~as
$$
(U\xi)(g)=\tilde\pi(g)\xi(g).
$$
This gives the result.
\ep

In particular, it follows that $W^*(G;\bar \omega)$ is a type I factor.
As $JW^*(G;\omega)J=W^*(G;\bar\omega)$, the von Neumann algebra $W^*(G;\omega)$ is a type I factor as well. By De Commer's duality result~\cite{DC2} we conclude that the map associating $[\omega_\pi]$ to the isomorphism class of $(B(H),\Ad\pi)$ is injective and its image is contained in the set of cohomology classes $[\omega]$ such that $W^*(G;\omega)$ is a type I factor.

Consider now an arbitrary cohomology class $[\omega]\in H^2(G;\T)$ such that $W^*(G;\omega)$ is a type~I factor. To finish the proof it suffices to establish the following.

\begin{lemma} \label{lem:Galois}
The pair $(W^*(G,\omega),\Ad\lambda^\omega)$ is a $G$-Galois object.
\end{lemma}

\bp Let us start with the I-factorial $\hat G$-Galois object $(W^*(G;\bar \omega),\alpha)$, with the coaction $\alpha$ given by~\eqref{eq:coaction}. By definition, the crossed product $W^*(G,\bar\omega)\rtimes_\alpha\hat G$ is the von Neumann subalgebra of $B(L^2(G)\otimes L^2(G))$ generated by $\alpha(W^*(G;\bar\omega))$ and $1\otimes L^\infty(G)$. The unitary $\omega_{21}\hat W$ commutes with $1\otimes L^\infty(G)$ and satisfies
$$
\omega_{21}\hat W(\lambda^{\bar\omega}_g\otimes\lambda_g)=(1\otimes\lambda^{\bar\omega}_g)\omega_{21}\hat W.
$$
Therefore the conjugation by this unitary defines an isomorphism of $W^*(G,\omega)\rtimes_\alpha\hat G$ onto the algebra $1\otimes B(L^2(G))$. This isomorphism maps $\alpha(W^*(G;\bar\omega))$ onto $1\otimes W^*(G;\bar\omega)$.

In order to understand what happens with the dual action, initially defined by the automorphisms $\Ad(1\otimes\rho_g)$, it is convenient to assume that the cocycle $\omega$ satisfies
\begin{equation*}\label{eq:cocycle-normalization}
\omega(g,e)=\omega(e,g)=\omega(g,g^{-1})=1\ \ \text{for all}\ \ g\in G,
\end{equation*}
which is always possible to achieve by replacing $\omega$ by a cohomologous cocycle. Then $\hat J$ is the modular conjugation of $W^*(\hat G;\bar\omega)$, so that $W^*(\hat G;\bar\omega)'=\hat JW^*(\hat G;\bar\omega)\hat J$. Put $\rho^{\omega}_g=\hat J\lambda^{\bar\omega}_g\hat J$. Then
$$
(\rho^{\bar\omega}_gf)(h)=\Delta(g)^{1/2}\omega(g,g^{-1}h^{-1})f(hg)=\Delta(g)^{1/2}\omega(h,g)f(hg).
$$
It is then not difficult to check that
$$
\omega_{21}\hat W(1\otimes\rho_g)\hat W^*\bar\omega_{21}=(\lambda^{\bar\omega}_g\otimes\rho^\omega_g).
$$

We thus see that the von Neumann algebra $\alpha(W^*(G;\bar\omega))'\cap (W^*(G;\bar\omega)\rtimes_\alpha \hat G)$, together with the restriction of the dual action, is isomorphic
to $\rho^{\omega}(G)''=\hat J W^*(G;\bar\omega)\hat J$ with the action given by the automorphisms $\Ad\rho^{\omega}_g$. As
$$
\rho^\omega_g=\hat J\lambda^{\bar\omega}_g\hat J=\hat JJ\lambda^{\omega}_gJ\hat J,
$$
we conclude that the $G$-Galois object associated with the $\hat G$-Galois object  $(W^*(G;\bar \omega),\alpha)$ is isomorphic to $(W^*(G;\omega),\Ad\lambda^\omega)$.
\ep

\subsection{Dual cocycles}\label{ss:dual-cocycles}

An important class of $G$-Galois objects arises from dual $2$-cocycles. By a \emph{dual unitary $2$-cocycle} on $G$ we mean a unitary element $\Omega\in W^*(G)\bar\otimes W^*(G)$ such that
\begin{equation}\label{eq:dual-cocycle}
(\Omega\otimes1)(\Dhat\otimes\iota)(\Omega)=(1\otimes\Omega)(\iota\otimes\Dhat)(\Omega).
\end{equation}
Similarly to the $\hat G$-Galois objects $W^*(G;\omega)$ considered above, such cocycles lead to $G$-Galois objects $W^*(\hat G;\Omega)$ (which for notational consistency with $W^*(G;\omega)$ we should have rather denoted by $W^*(\hat G;\Omega^*)$). Following~\cite[Section~4]{NT}, they can be described as follows.

Identify as usual the Fourier algebra $A(G)$ with the predual of $W^*(G)$. Given a dual unitary $2$-cocycle $\Omega\in W^*(G)\bar\otimes W^*(G)$, the von Neumann algebra $\mathcal N:=W^*(\hat G;\Omega)\subset B(L^2(G))$ is generated by the operators
$$
\pi_\Omega(f):=(f\otimes\iota)(\hat W\Omega^*),\quad f\in A(G).
$$
Define an action $\beta$ of $G$ on $\N$ by
$$
\beta_g(x):=(\Ad \rho_g)(x),\quad x\in\mathcal N,
$$
where we remind that $\rho_g=\hat J\lambda_g\hat J$ is the right regular representation.
The map $\pi_\Omega$ is the representation of the algebra $A(G)$ equipped with the new product
\begin{equation}\label{eq:star-Omega}
(f_1\star_\Omega  f_2)(g):=(f_1\otimes f_2)\big(\hat\Delta(\lambda_g)\Omega^*\big).
\end{equation}
The representation $\pi_\Omega$ has the equivariance property
\begin{equation}\label{eq:equivariance}
\beta_g\big(\pi_\Omega(f)\big)=\pi_\Omega(\lambda_gf).
\end{equation}

By \cite[Section~1.3]{VV}, the canonical weight $\tilde\varphi$ on $\N$ has the following description. Its GNS-space can be identified with $L^2(G)$, with the GNS-map $\tilde\Lambda\colon\mathfrak N_{\tilde\varphi}\to L^2(G)$ uniquely determined~by
\begin{equation}\label{eq:GNS0}
\tilde\Lambda(\pi_\Omega(f))=\check f\ \ \text{for} \ \ f\in A(G)\ \ \text{such that}\ \ \check f\in L^2(G),
\end{equation}
where we remind that $\check f(g)=f(g^{-1})$.
In particular, for $f$ as above we have
\begin{equation}\label{eq:tildephi}
\tilde\varphi(\pi_\Omega(f)^*\pi_\Omega(f))=\|\check f\|^2_2.
\end{equation}

\smallskip

Two dual unitary $2$-cocycles $\Omega$, $\Omega'$ are called cohomologous if there exists a unitary $u\in W^*(G)$ such that
$$
\Omega'=(u\otimes u)\Omega\Dhat(u)^*.
$$
The cohomology classes form a set $H^2(\hat G;\T)$. In general this set does not have any extra structure.

\begin{proposition}\label{prop:Galois-iso}
Two dual unitary $2$-cocycles $\Omega$, $\Omega'$ on a locally compact group  $G$ are cohomologous if and only if they define isomorphic $G$-Galois objects.
\end{proposition}

\bp
As $\Dhat(u)=\hat W^*(1\otimes u)\hat W$, it is easy to see that if $\Omega'=(u\otimes u)\Omega\Dhat(u)^*$ then $\Ad u$ defines a $G$-equivariant isomorphism of $W^*(\hat G;\Omega)$ onto $W^*(\hat G;\Omega')$.

Conversely, assume we have a $G$-equivariant isomorphism $\theta\colon W^*(\hat G;\Omega)\to W^*(\hat G;\Omega')$. Denote by $\tilde\Lambda$ and $\tilde\Lambda'$ the GNS-maps for these objects as described above. Then the isomorphism $\theta$ is implemented by the unitary $u$ defined by $u\tilde\Lambda(x)=\tilde\Lambda'(\theta(x))$. Since by~\eqref{eq:equivariance} and~\eqref{eq:GNS0} the actions of $G$ are implemented in a similar way by the unitaries $\rho_g$:
$$
\rho_g\tilde\Lambda(x)=\Delta(g)^{1/2}\tilde\Lambda(\beta_g(x))\ \ \text{and}\ \ \rho_g\tilde\Lambda'(x')=\Delta(g)^{1/2}\tilde\Lambda'(\beta'_g(x')),
$$
we conclude that $u\rho_g=\rho_g u$, hence $u\in W^*(G)$.

Denote by $\G$ and $\G'$ the corresponding Galois maps. Then by definition we have $(1\otimes u)\G=\G'(u\otimes u)$. On the other hand, by~\cite[Proposition~5.1]{DC} we have $\G=\hat W\Omega^*$ and $\G'=\hat W{\Omega'}^*$. Hence
$$
(1\otimes u)\hat W\Omega^*=\hat W{\Omega'}^*(u\otimes u).
$$
Using again that $\Dhat(u)=\hat W^*(1\otimes u)\hat W$, we conclude that $\Omega'=(u\otimes u)\Omega\Dhat(u)^*$.
\ep

Combined with Theorem~\ref{thm:I-Galois} this proposition allows one to describe a part of $H^2(\hat G;\T)$ in terms of cohomology of $G$. Namely, denote by $H^2_{\mathrm{I}}(\hat G;\T)$ the subset of $H^2(\hat G;\T)$ formed by the classes $[\Omega]$ such that $W^*(\hat G;\Omega)$ is a type I factor. Given such an $\Omega$, we can identify $W^*(\hat G;\Omega)$ with $B(H)$ for a Hilbert space $H$. Then the action $\beta$ of $G$ on $W^*(\hat G;\Omega)$ is given by a projective representation $\pi$ of $G$ on $H$, and we denote by $c_\Omega$ the corresponding $2$-cocycle $\omega_\pi$ on $G$. Similarly, denote by $H^2_{\mathrm{I}}(G;\T)$ the subset of $H^2(G;\T)$ formed by the classes $[\omega]$ such that $W^*(G;\omega)$ is a type I factor.

\begin{corollary}\label{cor:I-cocycles}
For any second countable locally compact group $G$, the map $\Omega\mapsto[c_\Omega]$ defines an embedding of $H^2_{\mathrm{I}}(\hat G;\T)$ into $H^2_{\mathrm{I}}(G;\T)$.
\end{corollary}

A natural question is whether this embedding is onto. We do not know the answer, but as a step towards a solution of this problem let us explain how dual cocycles arise from Galois maps under extra assumptions.

It will be useful to go beyond Galois objects. Assume we are given a square-integrable irreducible projective representation~$\pi$ of~$G$ on~$H$. Assume also that we are given a unitary map
\begin{equation}\label{eq:op}
\Op\colon L^2(G)\to \HS(H)\ \ \text{such that}\ \ \Op(\lambda_gf)=(\Ad\pi(g))(\Op(f)),
\end{equation}
which we will call a \emph{quantization map}. Here $\HS(H)$ denotes the Hilbert space of Hilbert--Schmidt operators on $H$.
As in Section~\ref{ss:HG}, consider the Duflo--Moore operator $K$ on $H$ and the weight $\tilde\varphi=\Tr(K^{1/2}\cdot K^{1/2})$.
As the GNS-space for $\tilde\varphi$ we could take $\HS(H)$, with the GNS-map $\tilde\Lambda$ uniquely determined by $\tilde\Lambda(TK^{-1/2}):=T$ for $T\in\HS(H)$ such that $TK^{-1/2}$ is a bounded operator. But using the unitary $\Op$ we can transport everything to $L^2(G)$. Thus we take $L^2(G)$ as the GNS-space, with the GNS-map uniquely determined by
\begin{equation}\label{eq:GNS}
\tilde\Lambda(\Op(f)K^{-1/2}):=f\ \ \text{for}\ \ f\in L^2(G)\ \ \text{such that}\ \  \Op(f)K^{-1/2}\in B(H).
\end{equation}
Consider the corresponding Galois map $\G$, so $\G\colon L^2(G)\otimes L^2(G)\to L^2(G)\otimes L^2(G)$,
\begin{equation}\label{eq:Galois}
\G(f_1\otimes f_2)(g,h)=\tilde\Lambda\Big(\big(\Ad\pi(g)\big)(\Op(f_1)K^{-1/2})\Op(f_2)K^{-1/2}\Big)(h).
\end{equation}
Finally, define
\begin{equation}\label{eq:Omega}
\Omega:=(\CJ\otimes\CJ)\G^*(1\otimes\CJ)\hat W,
\end{equation}
where we remind that $\CJ= J\hat J$.

\begin{proposition}\label{prop:dual-cocycle}
With the above setup and notation, the operator $\Omega$ is coisometric. It lies in the algebra $W^*(G)\bar\otimes W^*(G)$ and satisfies the cocycle identity~\eqref{eq:dual-cocycle}.

In particular, $\Omega$ is a dual unitary $2$-cocycle on $G$ if and only if $(B(H),\Ad\pi)$ is a $G$-Galois object. Moreover, if $\Omega$ is indeed unitary, then $(B(H),\Ad\pi)$ is isomorphic to the $G$-Galois object $(W^*(\hat G;\Omega),\beta)$ defined by $\Omega$.
\end{proposition}

\bp
Since the Galois maps are always isometric, it is clear that $\Omega$ is coisometric.

Next, a straightforward application of definition~\eqref{eq:Galois} together with scaling property~\eqref{eq:scaling2} yield the following identities for $\G$, cf.~\cite[Lemma~3.2]{DC}:
$$
\G(\lambda_g\otimes1)=(\rho_g\otimes1)\G,\ \ \G(1\otimes\lambda_g)=(\lambda_g\otimes\lambda_g)\G.
$$
Together with the identities
$$
\hat W(\rho_g\otimes1)=(\rho_g\otimes1)\hat W,\ \ \hat W(1\otimes\rho_g)=(\lambda_g\otimes\rho_g)\hat W,\ \ \CJ\rho_g=\lambda_g\CJ
$$
this implies that $\Omega$ commutes with the operators $\rho_g\otimes1$ and $1\otimes\rho_g$. Hence $\Omega\in W^*(G)\bar\otimes W^*(G)$.

Turning to the cocycle identity, by~\cite[Proposition~3.5]{DC} the Galois map $\G$ (denoted by $\tilde G$ in op.~cit.) satisfies the following hybrid pentagon relation:
$$
\hat W_{12}\G_{13}\G_{23}=\G_{23}\G_{12}.
$$
(More precisely, the result in~\cite{DC} is formulated only for the Galois objects, but an inspection of the proof shows that it remains valid for arbitrary integrable ergodic actions.)
Plugging in $\G=(1\otimes\CJ)\hat W\Omega^*(\CJ\otimes\CJ)$ we get
$$
\hat W_{12}\hat W_{13}\Omega^*_{13}\hat W_{23}\Omega^*_{23}=\hat W_{23}\Omega^*_{23}\hat W_{12}\Omega^*_{12},
$$
and using $\Dhat(x)=\hat W^*(1\otimes x)\hat W$ and the pentagon relation $\hat W_{12}\hat W_{13}\hat W_{23}=\hat W_{23}\hat W_{12}$ we obtain the required cocycle identity.

Finally, by the definition of $\Omega$, the Galois map $\G$ is unitary if and only if $\Omega$ is unitary. Assuming that $\Omega$ and $\G$ are unitary, by~\cite[Proposition~3.6(1)]{DC} the elements $(\omega\otimes\iota)(\G)$ for $\omega\in B(L^2(G))_*$ span a $\sigma$-weakly dense subspace of $\pi_{\tilde\varphi}(B(H))$. Recalling the definition of $W^*(\hat G;\Omega)$ we conclude that
$$
\pi_{\tilde\varphi}(B(H))=\CJ W^*(\hat G;\Omega)\CJ.
$$
The action of $G$ on $B(H)$ is implemented on the GNS-space by the unitaries $\lambda_g$, while that on $W^*(\hat\G;\Omega)$ by the unitaries $\rho_g$. Since $\CJ\rho_g\CJ=\lambda_g$, we see that the Galois objects $(B(H),\Ad\pi)$ and $(W^*(G;\Omega),\beta)$ are indeed isomorphic.
\ep

\begin{remark}
Although \cite[Proposition~3.6(1)]{DC} is formulated only for the Galois objects, its proof remains valid for any integrable ergodic action. Therefore we see that starting from a square-integrable irreducible projective representation $\pi$ of $G$ on $H$ and a unitary quantization map $\Op\colon L^2(G)\to \HS(H)$, we can define a dual coisometric cocycle $\Omega$ by~\eqref{eq:Omega} and then the set of elements $(\omega\otimes\iota)(\hat W\Omega^*)$ for $\omega\in B(L^2(G))_*$ span a $\sigma$-weakly dense subspace of $\CJ\pi_{\tilde\varphi}(B(H))\CJ$. Therefore $\Omega$ contains a complete information about $(B(H),\Ad\pi)$ independently of whether we deal with a Galois object or not.
\end{remark}

\begin{remark}\label{rem:two-star}
The quantization map $\Op$ defines a product $\star$ on $L^2(G)$ by
$$
\Op(f_1\star f_2)=\Op(f_1)\Op(f_2).
$$
In general the product $\star_\Omega$ defined by $\Omega$ is apparently {\em not} the same as $\star$ on $A(G)\cap L^2(G)$. In order to see this, let us proceed a bit informally, without trying to fully justify every step. By definition we have
$$
\CJ\pi_\Omega(f)\CJ=\CJ(f\otimes\iota)(\hat W\Omega^*)\CJ=(f\otimes\iota)(\G(\CJ\otimes1)).
$$
Take functions $f_1,f_2\in L^2(G)$ and consider the function $f\in A(G)$ defined by $f(g)=(\lambda_gf_1,f_2)$. Applying $(\cdot f_1,f_2)$ to the first leg of $\G(\CJ\otimes1)$ we then get
$$
\CJ\pi_\Omega(f)\CJ=\int_G\big(\Ad \pi(g)\big)(\Op(\CJ f_1)K^{-1/2})\overline{f_2(g)}\,dg=\Op(\Delta^{-1/2}f)K^{-1/2}.
$$
We thus conclude that $\star_\Omega$ is related to the quantization map $\Op'$ defined by
$$
\Op'(f)=\Op(\Delta^{-1/2}f)K^{-1/2},
$$
so that $\Op'(f_1\star_\Omega  f_2)=\Op'(f_1)\Op'(f_2)$.

The products defined by the quantization maps $\Op$ and $\Op'$ coincide if and only if the map
$f\mapsto \Op^{-1}(\Op(\Delta^{-1/2}f)K^{-1/2})$ is an automorphism with respect to $\star$. In general we see no reason why this should be the case. But this can happen.
Observe that since $\Delta$ is the only positive measurable function $F$ on $G$ such that $\lambda_gF=\Delta(g)^{-1}F$, any reasonable extension of~$\Op$ to a class of functions including $\Delta^s$ should satisfy $\Op(\Delta^{1/2})=c K^{-1/2}$ for a constant $c>0$. Then $\Op^{-1}(\Op(\Delta^{-1/2}f)K^{-1/2})=c^{-1}(\Delta^{-1/2}f)\star\Delta^{1/2}$. From this we see that for the map $f\mapsto \Op^{-1}(\Op(\Delta^{-1/2}f)K^{-1/2})$ to be an automorphism it suffices to have the identities
\begin{equation}\label{eq:chi-star}
\Delta^s\star\Delta^t=\Delta^{s+t},\ \
\Delta^{-1/2}f=c\Delta^{-\alpha}\star f\star\Delta^{\alpha-1/2}
\end{equation}
for some $\alpha\in\R$. For the examples studied in this paper we will indeed have such identities, with $c=1$ and $\alpha=1/2$. On other hand, for the example studied in~\cite{BGNT1} (which is not a Galois object) we had $c=1$ and $\alpha=1/4$.
\end{remark}

We thus see that the problem of describing $H^2_{\mathrm I}(\hat G;\T)$ reduces to the following question: it is true that for any I-factorial Galois object $(B(H),\Ad\pi)$ there is a unitary quantization map~\eqref{eq:op}? This can be reformulated as a representation-theoretic problem as follows. Assume we are given a $2$-cocycle $\omega$ on $G$ such that $W^*(G;\omega)$ is a type I factor. We identify $W^*(G;\omega)$ with $B(H)$ and put $\pi_\omega(g):=\lambda^\omega_g\in B(H)$. Then $g\mapsto \pi_\omega(g)\otimes \pi^c_\omega(g)$ is a well-defined unitary representation of $G$ on $H\otimes\bar H$, where $\pi^c_\omega(g)\bar\xi:=\overline{\pi_\omega(g)\xi}$. Is this representation equivalent to the regular representation?

The answer is known to be ``yes'' for finite groups~\cite{Wa,Mov}. Indeed, the Galois map gives a unitary equivalence
$$
\pi_\omega\otimes\pi_\omega^c\otimes\eps_H\otimes\eps_{\bar H}\sim\rho\otimes\eps_H\otimes\eps_{\bar H},
$$
where $\eps_L$ denotes the trivial representation of $G$ on the Hilbert space $L$. This implies the required equivalence $\pi_\omega\otimes\pi_\omega^c\sim\lambda$ for finite groups $G$, but falls short of what we need for general~$G$.

\begin{remark}\label{rem:ax+b}
In order to stress that it can be dangerous to rely too much on analogies with the finite group case, note that for any square-integrable irreducible projective representation~$\pi$ of~$G$ on $H$, the Galois map always defines an embedding of the representation $\Ad\pi\otimes\eps_{\HS(H)}$ on $\HS(H)\otimes \HS(H)$ into $\rho\otimes\eps_{\HS(H)}$. It follows that for finite groups existence of a unitary quantization map~\eqref{eq:op} is equivalent to $(B(H),\Ad\pi)$ being a Galois object. This is certainly (but until recently unexpectedly!) not the case for general groups. For example, for the connected component $G$ of the $ax+b$ group over $\R$ there are two inequivalent infinite dimensional irreducible unitary representations. They are both square-integrable and both admit unitary quantization maps~\cite{BG}. But $G$ has no I-factorial Galois objects, since $H^2(G;\T)$ is trivial and $W^*(G)$ is the sum of two type I factors.
\end{remark}

\subsection{Dual cocycles defined by genuine representations}

We now turn to Galois objects $(B(H),\Ad\pi)$ defined by genuine representations.

\begin{theorem}\label{thm:genuine-rep}
For any nontrivial second countable locally compact group $G$, a (square-integrable, irreducible) unitary representation $\pi\colon G\to B(H)$ such that $(B(H),\Ad\pi)$ is a $G$-Galois object exists if and only if $W^*(G)$  is a type I factor. Moreover, if $\pi$ exists, then
\begin{enumerate}
\item[(i)] $\pi$ is unique up to unitary equivalence; explicitly, by identifying $W^*(G)$ with $B(H)$ we can take $\pi(g)=\lambda_g$;
\item[(ii)] the Galois object $(B(H),\Ad\pi)$ is defined by a dual unitary $2$-cocycle $\Omega$ on $G$.
\end{enumerate}
\end{theorem}

Recall that by Proposition~\ref{prop:Galois-iso} the cohomology class of $\Omega$ is determined by the isomorphism class of the corresponding Galois object. Therefore the above theorem shows that if $W^*(G)$ is a type I factor, then we get a canonical class $[\Omega]\in H^2(\hat G;\T)$. In terms of Corollary~\ref{cor:I-cocycles}, this class corresponds to the unit of $H^2(G;\T)$.

Note also that the condition that $W^*(G)$ is a type I factor implies that $G$ is neither compact nor discrete, so the situation described in the theorem is a purely analytical phenomenon.

\bp[Proof of Theorem~\ref{thm:genuine-rep}]
The first statement is an immediate consequence of Theorem~\ref{thm:I-Galois} applied to genuine representations and, correspondingly, to the trivial $2$-cocycle $\omega=1$ on $G$.
Furthermore, that theorem implies that $\pi$ is unique up to equivalence as a projective representation. Therefore to prove part (i) we only have to show the slightly stronger statement that $\pi$ is also unique up to equivalence as a genuine representation. Thus, we identify $W^*(G)$ with $B(H)$, take $\pi(g)=\lambda_g$, and we have to show that for any character $\eta\colon G\to\T$ the representations $\eta\pi$ and $\pi$ are equivalent. The representations $\eta\lambda$ and $\lambda$ are unitarily equivalent, e.g., by Fell's absorption principle. It follows that there exists an automorphism $\theta$ of $W^*(G)$ such that $\theta(\lambda_g)=\eta(g)\lambda_g$ for all $g$. As $\theta$ is an automorphism of $B(H)=W^*(G)$, it is unitarily implemented, which means exactly that $\eta\pi$ and $\pi$ are unitarily equivalent.

\smallskip

In order to prove part (ii), by our results in Section~\ref{ss:dual-cocycles} it suffices to show that the representation $\pi\otimes\pi^c$ is equivalent to the regular representation.

Since we can identify $W^*(G)$ with $B(H)$ in such a way that $\pi(g)=\lambda_g$, the representation~$\lambda$ is a multiple of $\pi$, so we can write $\lambda\sim\pi\otimes\eps_L$, where $\eps_L$ is the trivial representation on a separable Hilbert space $L$. Since $G$ is nontrivial, the Hilbert space $H$ must be infinite dimensional. But then the multiplicity of the square-integrable representation $\pi$ in $\lambda$ must be infinite as well~\cite{DM}, so the Hilbert space $L$ is infinite dimensional.

By passing to the conjugate representations we get
$$
\pi^c\otimes\eps_L\sim\lambda^c\sim\lambda\sim\pi\otimes\eps_L.
$$
This implies that the irreducible representations $\pi^c$ and $\pi$ are equivalent.

Next, using Fell's absorption principle we get
$$
\pi\otimes\pi\otimes\eps_L\sim\pi\otimes\lambda\sim\lambda\otimes\eps_H\sim\pi\otimes\eps_{L}\otimes\eps_H.
$$
From this we see that the representation $\pi\otimes\pi$ is a multiple of $\pi$, and in order to conclude that $\pi\otimes\pi\sim\lambda$ it suffices to show that the multiplicity of $\pi$ in $\pi\otimes\pi$ is infinite. In other words, we have to check that the commutant of $(\pi\otimes\pi)(G)$ in $B(H\otimes H)$ is infinite dimensional. Equivalently, that the commutant of $(\lambda\otimes\lambda)(G)$ in $W^*(G)\bar\otimes W^*(G)$ is infinite dimensional.

More generally, let us show that if $G_1$ is a closed nonopen subgroup of a second countable locally compact group $G_2$ such that $W^*(G_1)$ is a type I factor, then the relative commutant $W^*(G_1)'\cap W^*(G_2)$ is infinite dimensional.

Assume $W^*(G_1)'\cap W^*(G_2)$ is finite dimensional. Denote by $\Delta_i$ the modular function of~$G_i$, by $\mu_i$ the Haar measure on $G_i$ and by $\hat\varphi_i$ the standard Haar weight on $W^*(G_i)$. The modular group of $\hat\varphi_2$ is given by $\sigma_t(\lambda_g)=\Delta_2(g)^{it}\lambda_g$. It preserves $W^*(G_1)$, and since $W^*(G_1)$ is a type I factor, there exists a normal semifinite faithful weight $\hat\varphi_1'$ on $W^*(G_1)$ with the same modular group. Consider the unique normal semifinite operator-valued weight $P\colon W^*(G_2)\to W^*(G_1)$ such that $\hat\varphi'_1 P=\hat\varphi_2$. Since $W^*(G_1)$ is a type I factor and $W^*(G_1)'\cap W^*(G_2)$ is finite dimensional, it follows, e.g.,~from~\cite[Corollary~12.12]{Stratila} applied to ${\mathcal M}=W^*(G_1)$ that such an operator-valued weight must be bounded, hence it is a scalar multiple of a conditional expectation. In particular, the weight $\hat\varphi_2|_{W^*(G_1)}$ is semifinite. This, in turn, implies, that $\hat\varphi_2|_{W^*(G_1)}$ is a Haar weight on $W^*(G_1)$, hence $\hat\varphi_2|_{W^*(G_1)}=c\,\hat\varphi_1$ for a constant $c>0$. We can also conclude that $\Delta_2|_{G_1}=\Delta_1$.

Denote by $E$ the conditional expectation obtained by rescaling $P$, so that $c\,\hat\varphi_1 E=\hat\varphi_2$. Using the identity $\hat\varphi_2(xy)=c\,\hat\varphi_1(E(x)y)$ for appropriate elements $x\in W^*(G_2)$ and $y\in W^*(G_1)$, it is not difficult to compute $E$ on a dense set of elements. Namely, if $x=\int_{G_2}F(g)\lambda_g\, d\mu_2(g)$, with $F=f* f'$ (convolution in $L^1(G_2)$) for some $f,f'\in C_c(G_2)$, then we must have
$$
E(x)=c^{-1}\int_{G_1}F(g)\lambda_g\,d\mu_1(g).
$$
Applying this to functions $f'\ge0$ supported in arbitrarily small neighbourhoods of the identity and normalized by $\|f'\|_1=1$, which form an approximate unit in $L^1(G_2)$, and using the assumed continuity of $E$, we conclude that
$$
E\Big(\int_{G_2}f(g)\lambda_g\, d\mu_2(g)\Big)=c^{-1}\int_{G_1}f(g)\lambda_g\,d\mu_1(g)\ \ \text{for all}\ \ f\in C_c(G_2).
$$
But this formula certainly defines an unbounded map, so we get a contradiction. Indeed, take any nonzero function $\tilde f\in C_c(G_1)$. Since $G_1$ has zero measure in $G_2$, we can extend $\tilde f$ to a function $f\in C_c(G_2)$ with the same supremum-norm but supported in a set of arbitrarily small Haar measure, so that the $L^1$-norm of $f$ can be made arbitrarily small. Since the operator norm is dominated by the $L^1$-norm, we thus see that the preimage of $\int_{G_1}\tilde f(g)\lambda_g\,d\mu_1(g)$ under $E$ has elements of arbitrarily small norm.
\ep

As we have already observed, the class $[\Omega]\in H^2(\hat G;\T)$ corresponds to the unit of $H^2(G;\T)$, so one might wonder whether $\Omega$ is also a coboundary, that is, cohomologous to $1$. But this is surely not the case, since $W^*(\hat G;\Omega)\cong B(H)$, while $W^*(\hat G)=L^\infty(G)$. In fact, the following stronger nontriviality property holds.

\begin{proposition}\label{prop:nontriviality}
If $G$ is a nontrivial second countable locally compact group with group von Neumann algebra a factor of type I, and $\Omega$ is the dual unitary $2$-cocycle given by Theorem~\ref{thm:genuine-rep}, then the twisted locally compact quantum group $(W^*(G),\Omega\Dhat(\cdot)\Omega^*)$ is neither commutative nor cocommutative.
\end{proposition}

\bp
The algebra $W^*(G)$, being a nontrivial type I factor, is clearly noncommutative. This implies that the group $G$ is noncommutative. By Remark~\ref{rem:genuine-q} it follows that the quantum group $G_\Omega$ obtained by reflecting $G$ across the Galois object $(W^*(\hat G;\Omega),\beta)$ is a genuine quantum group, that is, the coproduct $\Dhat_\Omega$ on $W^*(G_\Omega)$ is not cocommutative. But $(W^*(G_\Omega),\Dhat_\Omega)$ is exactly $(W^*(G),\Omega\Dhat(\cdot)\Omega^*)$, see~\cite{DC}.
\ep

\subsection{Examples: subgroups of the affine group}

We now introduce the main class of examples studied in this paper.

Let $V$ be a nontrivial second countable locally compact abelian group, $Q$ be a second countable locally compact group of continuous automorphisms
of $V$ and call $G$ the semidirect product $Q\ltimes V\subset\Aff(V):=\Aut(V)\ltimes V$. It has the group law
$$
(q,v)(q',v')=(qq',v+qv'),\quad q,q'\in Q,\; v,v'\in V.
$$
The unit element is $(\id,0)$ and the inverse is $(q,v)^{-1}=(q^{-1},-q^{-1}v)$. Whenever convenient we identify $q\in Q$ with $(q,0)\in Q\ltimes V$ and $v\in V$ with $(\id,v)\in Q\ltimes V$. We will also usually write~$1$ instead of $\id$ for the unit of $Q$.

We denote by $q^\flat$ the dual action of $Q$ on $\hat V$ defined by the identity
$$
e^{i\langle q^\flat\xi,v\rangle}=e^{i\langle \xi,q^{-1}v\rangle},
$$
where we remind that $(\xi,v)\mapsto e^{i\langle \xi,v\rangle}$ simply denotes the duality paring $\hat V\times V\to\mathbb T$.

\smallskip

We will study the groups $Q\ltimes V$ satisfying the following property.

\begin{assumption}
\label{Fro}
There is an element $\xi_0\in\hat V$ such that the map
$$
\phi:Q\to\hat V,\quad q\mapsto q^\flat\xi_0,
$$
is a measure class isomorphism. In this case,
we say that $(V,Q)$  satisfies the {\bf dual orbit condition}.
\end{assumption}

We note in passing that the groups studied in~\cite{J} do \emph{not} satisfy this condition.

\begin{remark}
If $(V,Q)$ satisfies the  dual orbit condition, then $V$ cannot be compact, since the neutral element of $\hat V$ cannot belong to the $Q$-orbit of $\xi_0$ and therefore $\hat V$ cannot be discrete. 
Note also that the stabilizer of $\xi_0$ in $Q$ must be trivial, so that the map $\phi$ is injective.
\end{remark}

In order to give some concrete examples, let $\mathbb K$ be a nondiscrete locally compact skew-field.
If $\mathbb K$  is commutative then $\mathbb K$ is a local field. This means that $\mathbb K$  is isomorphic either to $\R,\C$, a finite degree extension of the field of $p$-adic numbers $\mathbb Q_p$ or to a field $\mathbb F_q((X))$
of Laurent series with coefficients in a finite field $\mathbb F_q$.
If $\mathbb K$  is a skew-field, then $\mathbb K$ is isomorphic to a   finite dimensional division algebra over a local field $k$. As an abelian group $\mathbb K$ is self-dual, with a pairing $\mathbb K\times \mathbb K\to\T$ given by $(a,b)\mapsto\chi(\Tr_{\mathbb K/k}(ab))$ for a nontrivial character $\chi$ of $k$. Denote by~$\chi_{\mathbb K}$ the character $\chi\circ\Tr_{\mathbb K/k}$.

\begin{example}
For  $n\ge1$, let
$$
V=\operatorname{Mat}_n(\mathbb K)\quad \mbox{and} \quad  Q=\operatorname{GL}_n(\mathbb K),
$$
with the action given by left  matrix multiplication.  Under the isomorphism $\hat V\simeq V$ associated with the pairing
 $(A,B)\mapsto \chi_{\mathbb K}({\rm Tr}(\prescript{t}{}\!AB))$, the dual action is  given by $(A,M)\mapsto \prescript{t}{}\!A^{-1}M$. Therefore,
 both $(V,Q)$ and $(\hat V,Q)$ satisfy the dual orbit condition.
\end{example}

\begin{example}
Let $\tau$ be any order-two ring automorphism of  $\operatorname{Mat}_n(\mathbb K)$.
Consider the quaternionic type group $\mathbb H_n^\pm(\mathbb K,\tau)$ given by the subgroup of
 ${\rm GL}_{2n}(\mathbb K)$ of elements  of the form
 $$
 \begin{pmatrix}
 A&B\\
\pm \tau(B)&\tau(A)
 \end{pmatrix},
 \quad A,B\in \operatorname{Mat}_n(\mathbb K).
 $$
Set
 $$
V=\operatorname{Mat}_n(\mathbb K)\oplus \operatorname{Mat}_n(\mathbb K) \quad \mbox{and} \quad Q=\mathbb H_n^\pm(\mathbb K,\tau).
$$
Here also,  both $(V,Q)$ and $(\hat V,Q)$ satisfy the dual orbit condition.
\end{example}

\begin{example} \label{ex:oo}
For $n\ge1$ and $m\ge 2$, let
$$
V=\underbrace{\operatorname{Mat}_n(\mathbb K)\oplus \cdots\oplus\operatorname{Mat}_n(\mathbb K)}_m \quad \mbox{and} \quad Q=
\begin{pmatrix}1&\cdots&0&0\\
\vdots& \ddots& \vdots&\vdots\\
0&\cdots&1&0\\
\operatorname{Mat}_n(\mathbb K)&\cdots&\operatorname{Mat}_n(\mathbb K)&{\rm GL}_n(\mathbb K)\end{pmatrix}\subset\operatorname{GL}_{nm}(\mathbb K).
$$
Then $(V,Q)$ satisfies the dual orbit condition, since for $M_1,\dots,M_m,B_1,\dots, B_{m-1}\in \operatorname{Mat}_n(\mathbb K)$ and $A\in {\rm GL}_n(\mathbb K)$ we have
$$
\begin{pmatrix}1&\cdots&0&0\\
\vdots& \ddots& \vdots&\vdots\\
0&\cdots&1&0\\
B_1&\cdots&B_{m-1}&A\end{pmatrix}^\flat
\begin{pmatrix}
M_1\\
\vdots\\
M_{m-1}\\
M_{m}
\end{pmatrix}=
\begin{pmatrix}
M_1-\prescript{t}{}B_1\prescript{t}{}A^{-1}M_m\\
\vdots\\
M_{m-1}-\prescript{t}{}B_{m-1}\prescript{t}{}A^{-1}M_m\\
\prescript{t}{}A^{-1}M_m
\end{pmatrix}.
$$
On the other hand, the dual pair $(\hat V,Q)$ does not satisfy the dual orbit condition, as there is no $Q$-orbit of full measure in $V$.
\end{example}

\begin{example}
Let $\mathbb A$ be a  nondiscrete second countable locally compact (unital, but not necessarily commutative) ring, such that the set  $\mathbb A^\times$ of invertible
elements is of full Haar measure. Then, the pair $(\hat{\mathbb A},\mathbb A^\times$) satisfies the dual orbit condition. See
\cite[Section~4]{BSV} for such explicit examples. Note that some of these examples show that the map $\phi\colon Q\to\hat V$ is not necessarily open.
\end{example}

Let us now say a few words about the case when $Q$ is a real or complex Lie group, $V$ is a finite dimensional vector space and the action of $Q$ on $V$ is given by a representation $\rho$. In this case we can identify $\hat V$ with $V^*$, and the action of $Q$ on $V^*$ is given by the contragredient representation $\rho^c$. If $(Q,V)$ satisfies the dual orbit condition, then $Q$ has the same dimension as $V$ and the map $\phi\colon Q\to V^*$ is open, so $(V^*,\rho^c(Q))$ is a \emph{prehomogeneous vector space}.

\begin{remark}\label{rem:ds}
In the complex case, assuming $\dim Q=\dim V$, the dual orbit condition is satisfied for $\xi_0\in V^*$ as long as $\xi_0$ has trivial stabilizer. Indeed, consider the set $\Omega$ of vectors $\xi\in V^*$ such that the map $\mathfrak{q}\to V^*$, $X\mapsto (d\rho^c)(X)\xi$, is a linear isomorphism. This set is nonempty, as $\xi_0\in\Omega$, and Zariski open in $V^*$. It follows that it is a dense, connected, open subset of $V^*$ in the usual topology. Since the $Q$-orbit of every element of $\Omega$ is open, it follows that $\Omega$ consists entirely of one orbit, so the dual orbit condition is satisfied. As a byproduct we see that $Q$ must be connected.

Note that these arguments do not apply in the real case, as a Zariski open subset of $\R^n$ can have finitely many connected components.
\end{remark}

By a repeated use of the following simple lemma we can construct more and more complicated examples.

\begin{lemma}\label{lem:ds}
Assume $(Q,V)$ satisfies the dual orbit condition, with $Q$ a complex Lie group and the action of $Q$ given by a representation $\rho\colon Q\to \operatorname{GL}(V)$. Then the adjoint action of the Lie group $G:=Q\ltimes V$ on its Lie algebra $\g$ also satisfies the dual orbit condition.
\end{lemma}

\bp
By Remark~\ref{rem:ds} it suffices to find a vector $\eta_0\in\g^*$ with trivial stabilizer. By assumption there exists $\xi_0\in V^*$ with trivial stabilizer in $Q$. Identifying $\g$ with $\mathfrak q\times V$ we let $\eta_0(X,w):=\xi_0(w)$. The adjoint action is given by
$$
\big(\Ad (q,v)\big)(X,w)=\Big((\Ad q)(X),\rho(q)w- (d\rho)\big((\Ad q)(X)\big)v\Big).
$$
Assume now that $(q,v)\in G$ lies in the stabilizer of $\eta_0$, that is,
$$
\xi_0(\rho(q)w)-\xi_0\Big((d\rho)\big((\Ad q)(X)\big)v\Big)=\xi_0(w)\ \ \text{for all}\ \ X\in\mathfrak q\ \ \text{and}\ \ w\in V.
$$
Taking $X=0$, we see that $q$ stabilizes $\xi_0$, hence $q=e$. Since the map $\mathfrak{q}\to V^*$, $X\mapsto (d\rho^c)(X)\xi_0$, is an isomorphism, we then conclude that $v=0$.
\ep

This lemma and its proof show, both in real and complex cases,  that if $(Q,V)$ satisfies the dual orbit condition, then the Lie algebra of $G:=Q\ltimes V$ is \emph{Frobenius}, meaning that $G$ has an open coadjoint orbit, or equivalently, there exists $\eta_0\in\g^*$ such that the bilinear form $\eta_0([X,Y])$ on $\g$ is nondegenerate. This has already been observed in~\cite{Oo}. Moreover, the converse is almost true: by~\cite[Theorem~4.1]{Oo}, if $\dim Q=\dim V$ and the Lie algebra of $Q\ltimes V$ is Frobenius, then there exists $\xi_0\in V^*$ such that the map $\mathfrak{q}\to V^*$, $X\mapsto (d\rho^c)(X)\xi_0$, is a linear isomorphism. This is a bit less than what we need since the stabilizer of $\xi_0$ can still be a nontrivial discrete subgroup of $Q$ and, in the real case, the open set $\rho^c(Q)\xi_0\subset V^*$ can be nondense.

\smallskip

In addition to Lemma~\ref{lem:ds}, another way of producing new examples is to start with a pair $(Q,V)$ satisfying the dual orbit condition and multiply the representation of $Q$ on $V$ by a quasi-character of $Q$. This can destroy the dual orbit condition, but not necessarily.

\begin{example}
Consider $Q=\C^*$, $V=\C$, with $Q$ acting on $V$ by multiplication. Obviously, the dual orbit condition is satisfied. By Lemma~\ref{lem:ds} the adjoint action of the $ax+b$ group $G:=Q\ltimes V$ satisfies the dual orbit condition as well. This is basically Example~\ref{ex:oo} for $n=1$ and $m=2$. If we multiply the adjoint representation of $G$ by the quasi-character $(q,v)\mapsto q^k$, $k\in\Z$, we get the representation
$$
\rho_k\colon G\to\operatorname{GL}_2(\C),\quad \rho_k(q,v)=\begin{pmatrix} q^k & 0\\ -q^kv & q^{k+1} \end{pmatrix}.
$$
The action of $G$ on $\C^2$ defined by $\rho_k$ has a dense open dual orbit if and only if $k\ne-1$. The dual orbit condition is satisfied if and only if $k=0$ or $k=-2$, while for other $k\ne-1$ the stabilizers of the points on the dense dual orbit are finite and nontrivial. We can use the same formulas in the real case. Then the dual orbit condition is satisfied if and only if $k$ is even.

We remark in passing that for $k=1$ the group $G\ltimes_{\rho_1}\C^2$ is an extension of $\C^*$ by the Heisenberg group. It is used in the construction of an extended Jordanian twist in~\cite{KLM}.
\end{example}

Returning to the general case, we have the following result.

\begin{proposition}\label{prop:G-type-I}
If a pair $(V,Q)$ satisfies the dual orbit condition, then the group von Neumann algebra of $G=Q\ltimes V$ is a type I factor.
\end{proposition}

\bp
By conjugating by the partial Fourier transform $\CF_V$ we get $W^*(G)\cong L^\infty(\hat V)\rtimes Q$. But by the dual orbit condition the space $\hat V$, considered as a measure space equipped with an action of $Q$, can be identified with $Q$ equipped with the action of $Q$ by left translations. Hence $W^*(G)\cong B(L^2(Q))$.
\ep

By Theorem~\ref{thm:genuine-rep} it follows that we get a canonical class $[\Omega]\in H^2(\hat G;\T)$ defined by the Galois object $(W^*(G),\Ad\lambda)$. In order to get an explicit representative of this class we need to fix a (unique up to equivalence) representation~$\pi$ as in that theorem and choose a unitary quantization map~\eqref{eq:op}. This is the task we are going to undertake in the next section.

\section{Kohn--Nirenberg quantization}\label{sec:KN}

\subsection{Kohn--Nirenberg quantization of \texorpdfstring{$V\times\hat V$}{VxhatV}}

Quantization of the abelian self-dual group $V\times\hat V$ has been considered long ago \cite{Segal, Weil}. Here we consider the Kohn--Nirenberg
 quantization. It has some benefit compared with the Weyl type quantization
for which one has to assume the map $V\to V$, $v\mapsto 2v$, to be a homeomorphism.

The Kohn--Nirenberg quantization can initially be defined as the continuous injective linear map
\begin{align*}
\OKN:\mathscr{S}'(V\times\hat V)\longrightarrow \mathscr L\big(\mathscr{S}(V),\mathscr{S}'(V)\big),
\end{align*}
from tempered Bruhat distributions on $V\times\hat V$ to
continuous linear operators from the Bruhat-Schwartz space on $V$ to  tempered Bruhat distributions on $V$ (see \cite[Section 9]{Bruhat} for a precise definition).
It is defined by the formula
$$
\OKN(F)\vf(v):=\int_{V\times\hat V} e^{i \langle\xi,v-v'\rangle}\,F(v',\xi)\,\vf(v')\,dv'd\xi,\quad F\in\mathscr{S}'(V\times\hat V),\; \vf\in \mathscr{S}(V).
$$

\begin{remark}
\label{AKNQ}
This quantization map should rather be called anti-Kohn--Nirenberg. The Kohn--Nirenberg one is given by the formula
$$
\int_{V\times\hat V} e^{i \langle\xi,v-v'\rangle}\,F(v,\xi)\,\vf(v')\,dv'd\xi.
$$
The two quantization maps are unitarily  equivalent and for our purposes it is easier to work with anti-Kohn--Nirenberg.
\end{remark}

The distributional kernel of the operator $\OKN(F)$ is therefore  given by
\begin{align}
\label{KO}
(v,v')\mapsto\big((1\otimes\CF_V^*) F\big)(v',v-v').
\end{align}
Since an operator on $L^2(V)$ with kernel $K$ is Hilbert--Schmidt if and only if $K\in L^2(V\times V)$, we immediately deduce:
\begin{lemma}
\label{Unit}
The Kohn--Nirenberg quantization map $\OKN$ extends to a unitary isomorphism from $L^2(V\times\hat V)$ onto $\HS(L^2(V))$.
\end{lemma}
Hence, the Hilbert space $L^2(V\times\hat V)$ can be endowed with an associative product
\begin{align}
\label{S0}
f_1 \star _0 f_2:=\OKN^*\big(\OKN(f_1)\OKN(f_2)\big).
\end{align}

Very important here are the symmetries of the Kohn--Nirenberg product $\star _0$. We have an action of $\Aff(V)$ on $V\times \hat V$:
\begin{align}
\label{actW}
{\rm Aff}(V)\times (V\times \hat V)\to V\times \hat V,\quad (q,v).(w,\xi)=(qw+v,q^\flat\xi).
\end{align}
With $|\cdot|_V$ the modulus function of ${\rm Aut}(V)$ and $|\cdot|_{\hat V}$ the one of ${\rm Aut}(\hat V)$, we observe that
$$
|q^\flat|_{\hat V}=|q|_V^{-1}.
$$
We will usually use only the modulus $|q|_V$ in various formulas and write it simply as $|q|$.
The action~\eqref{actW} gives rise to a unitary representation $\pi_{V\times\hat V}$ of ${\rm Aff}(V)$ on $L^2(V\times\hat V)$ defined~by
\begin{align}
\label{eq:repVVhat}
\pi_{V\times\hat V}(q,v)\vf(w,\xi)=\vf(q^{-1}(w-v),{q^{-1}}^\flat\xi).
\end{align}
We can also define a unitary representation $\pi_V$ of $\Aff(V)$ on $L^2(V)$ by
\begin{equation}\label{eq:repV}
(\pi_V(q,v)\vf)(w)=|q|^{-1/2}\vf(q^{-1}(w-v)).
\end{equation}

\begin{lemma}\label{lem:equivariance}
For all $f\in L^2(V\times\hat V)$ and $(q,v)\in\Aff(V)$, we have
$$
\OKN(\pi_{V\times\hat V}(q,v)f)=\pi_V(q,v)\OKN(f)\pi_V(q,v)^*.
$$
In particular, the operators of the representation  $\pi_{V\times\hat V}$ of ${\rm Aff}(V)$ act on $L^2(V\times \hat V)$ by algebra automorphisms for
$ \star _0$.
\end{lemma}

\begin{proof}
Take $\vf\in L^2(V)$. Then

\medskip
$\displaystyle \big(\OKN(\pi_{V\times\hat V}(q,v)(f))\pi_V(q,v)\vf\big)(w)$
\smallskip
\begin{align}
&=\int_{V\times\hat V}e^{i\langle\xi,w-w'\rangle}(\pi_{V\times\hat V}(q,v)f)(w,\xi)(\pi_V(q,v)\vf)(w')\,dw'd\xi\nonumber\\
&=|q|^{-1/2}\int_{V\times\hat V}e^{i\langle\xi,w-w'\rangle}f(q^{-1}(w'-v),{q^{-1}}^\flat\xi)\vf(q^{-1}(w'-v))\,dw'd\xi.\label{eq:int22}
\end{align}
On the other hand,
\begin{align*}
\big(\pi_V(q,v)\OKN(f)\vf\big)(w)&=|q|^{-1/2}\big(\OKN(f)\vf\big)(q^{-1}(w-v))\\
&=|q|^{-1/2}\int_{V\times\hat V}e^{i\langle\xi,q^{-1}(w-v)-w'\rangle}f(w',\xi)\vf(w')\,dw'd\xi.
\end{align*}
Using that
$$
e^{i\langle\xi,q^{-1}(w-v)-w'\rangle}=e^{i\langle q^\flat\xi,w-v\rangle}e^{i\langle q^\flat\xi,-qw'\rangle}
$$
and that the Haar measure $dw'd\xi$ of $V\times \hat V$ is invariant under the transformations  $(w',\xi)\mapsto(qw',q^\flat\xi)$,
we can write the last integral as
$$
|q|^{-1/2}\int_{V\times\hat V}e^{i\langle \xi,w-v\rangle}e^{i\langle \xi,-w'\rangle}f(q^{-1}w',{q^{-1}}^\flat\xi)\vf(q^{-1}w')\,dw'd\xi,
$$
and this is equal to~\eqref{eq:int22} by translation invariance of the Haar measure $dw'$ on $V$.
\end{proof}

\subsection{Dual cocycles on subgroups of \texorpdfstring{$\Aff(V)$}{Aff(V)} }

Let $Q\ltimes V\subset \Aff(V)$ be a subgroup satisfying the dual orbit Assumption~\ref{Fro}.
Let $d_Q(q)$ be a Haar measure on $Q$ and $\Delta_Q$  the
modular function. Routine computations show that
a left invariant Haar  measure and modular function on $G$ are respectively given by
$$
d(q,v)=\frac{d_Q(q)\,dv}{|q|}\quad\mbox{and}\quad \Delta(q,v)=\frac{\Delta_Q(q)}{|q|}.
$$

The measure class isomorphism:
\begin{align}
\label{TM}
G\to V\times \hat V:(q,v) \mapsto (q,v).(0,\xi_0)=(v,q^\flat \xi_0),
\end{align}
intertwines the left action of $G$ on itself with the  action \eqref{actW} on $V\times \hat V$.
Since  the  Haar measure on $V\times \hat V$ is invariant under the affine action of $G$ given in \eqref{actW},
the pull-back by the map~\eqref{TM} of the Haar measure on $V\times \hat V$
defines a nonzero left invariant Radon measure on $G$ and thus is a scalar multiple  of the Haar measure of $G$. We therefore may assume that the Haar measure on~$Q$ is normalized so that the $G$-equivariant linear operator $\tilde U_\phi\colon L^2(V\times \hat V)\to L^2(G)$ defined~by
\begin{equation}\label{eq:tilde-U-phi}
(\tilde U_\phi f)(q,v):=f\big(v,\phi(q) \big),
\end{equation}
where $\phi(q)=q^\flat\xi_0$, is unitary. Equivalently, our normalization is such that the Haar measure~$d\xi$ on~$\hat V$ is the push-forward under $\phi$ of the measure $|q|^{-1}d_Q(q)$ on $Q$, that is,
\begin{equation}\label{eq:dqdxi}
\int_Q f(\phi(q))\frac{d_Q(q)}{|q|}=\int_{\hat V}f(\xi)\,d\xi\ \ \text{for all}\ \ f\in L^1(\hat V).
\end{equation}
Note in passing that this uniquely determines the Haar measure on $G$: if we multiply $dv$ by a scalar, then~\eqref{eq:dqdxi} and unitarity of $\CF_V$ force us to divide $d_Q(q)$ by the same scalar.

Therefore we get a unitary quantization map
\begin{align}
\label{Op}
\Op\colon L^2(G)\to \HS(L^2(V)),\ \ \Op(f)=\OKN(\tilde U_\phi^* f),
\end{align}
and if we denote by $\pi$ the canonical representation of $G\subset\Aff(V)$ on $L^2(V)$, that is, $\pi={\pi_V}|_G$, then by Lemma~\ref{lem:equivariance} we have
$$
\Op(\lambda_g f)=\pi(g)\Op(f)\pi(g)^*.
$$

\begin{lemma}\label{lem:DM}
The representation $\pi$ of $G$ on $L^2(V)$ is irreducible, square-integrable and the associated Duflo--Moore formal degree operator is
$K=\CF_V^*\,M(\Delta^{-1}\circ\phi^{-1})\,\CF_V$.
\end{lemma}

Since $\Delta$ is a $V$-invariant function on $G$, it is naturally viewed as a function on $Q$. Therefore by $M(\Delta^{-1}\circ\phi^{-1})$ we mean the operator of multiplication by the function $\xi\mapsto \Delta(\phi^{-1}(\xi))^{-1}=|\phi^{-1}(\xi)|/\Delta_Q(\phi^{-1}(\xi))$.

\bp
It is more convenient to work on $L^2(Q)$ with the equivalent representation $\tilde\pi$ given by
$$
\tilde\pi(q,v):=\mathcal V\,\pi(q,v)\, \mathcal V^*,
$$
where $\mathcal V:L^2(V)\to L^2(Q)$, $(\mathcal V\vf)(q):=|q|^{-1/2}(\CF_V\vf)(q^\flat\xi_0)$. A quick computation shows that
$$
(\tilde\pi(q,v)\vf)(q')=e^{-i\langle\phi(q'),v\rangle}\,\vf(q^{-1}q').
$$

The restriction of the representation $\pi$ to $V$ is simply the regular representation. It follows that any operator in $B(L^2(V))$ commuting with $\pi(G)$ must belong to $\pi(V)''$. Passing to the equivalent representation $\tilde\pi$, this means that any operator in $B(L^2(Q))$ commuting with $\tilde\pi(G)$ must be an operator of multiplication by a function in $L^\infty(Q)$. In addition this function must be invariant under the left translations on $Q$, hence it is constant. Thus $\tilde\pi$ is irreducible.

\smallskip

Turning to square-integrability, for $\vf_1,\vf_2\in C_c(Q)$, we have
\begin{align*}
(\tilde\pi(q,v)\vf_1,\vf_2)&=\int_Q e^{-i\langle\phi(q'),v\rangle}\,\vf_1(q^{-1} q')\overline{\vf_2(q')}\,d_Q(q')\\
&= \int_{\hat V} e^{-i\langle\xi,v\rangle}\,\vf_1\big(q^{-1}\phi^{-1}(\xi)\big)\,\overline{\vf_2\big(\phi^{-1}(\xi)\big)}\,|\phi^{-1}(\xi)|\,d\xi,
\end{align*}
with the second equality following from~\eqref{eq:dqdxi}. If we set
$$
f_q(\xi) :=  \vf_1\big(q^{-1}\phi^{-1}(\xi)\big)\,\overline{\vf_2\big(\phi^{-1}(\xi)\big)}\,|\phi^{-1}(\xi)|,
$$
then, for every fixed $q\in Q$, the function $f_q$ has compact essential support and is bounded. Hence $f_q\in L^2(\hat V)$ and
$$
(\tilde\pi(q,v)\vf_1,\vf_2) =(\mathcal F_{\hat V}f_q)(v).
$$
The Plancherel formula for $V$ gives then
\begin{align*}
\int_G |(\tilde\pi(q,v)\vf_1,\vf_2)|^2\, \frac{d_Q(q)dv}{|q|}&=
\int_Q\Big(\int_V |(\mathcal F_{\hat V}f_q)(v)|^2\,dv\Big)\frac{d_Q(q)}{|q|}\\
&=\int_Q\Big(\int_{\hat V} |f_q(\xi)|^2\,d\xi\Big)\frac{d_Q(q)}{|q|}\\
&=\int_{Q\times Q} |\vf_1\big(q^{-1}\phi^{-1}(\xi)\big)|^2\,|\vf_2\big(\phi^{-1}(\xi)\big)|^2\,|\phi^{-1}(\xi)|^2\, \frac{d_Q(q)\,d\xi}{|q|}\\
&=\int_{Q\times Q}   |\vf_1(q^{-1}q')|^2\,|\vf_2(q')|^2\,|q^{-1}q'|\,d_Q(q)\,d_Q(q')\\
&=\int_{Q}   |\vf_1(q)|^2\,\frac{|q|}{\Delta_Q(q)}d_Q(q)\int_Q|\vf_2(q')|^2\,d_Q(q').
\end{align*}
This shows that $\tilde\pi$ is square-integrable and the Duflo--Moore operator is the operator of multiplication by the function $q\mapsto |q|/\Delta_Q(q)$.
This gives the result.
\ep

The next natural question is whether $(B(L^2(V)),\Ad\pi)$ is a $G$-Galois object, or equivalently, by Theorem~\ref{thm:genuine-rep} and Proposition~\ref{prop:G-type-I}, whether $\pi$ is quasi-equivalent to the regular representation. If it is, then we can construct a dual unitary $2$-cocycle on $G$ by the procedure described in Section~\ref{ss:dual-cocycles}. Instead of exactly following that procedure, however, we will construct this cocycle directly from the product $\star$ on $L^2(G)$ defined~by
\begin{equation}\label{eq:star0}
f_1 \star  f_2:=\Op^*(\Op(f_1)\Op(f_2))=\tilde U_\phi\big(\tilde U_\phi^*f_1  \star _0 \tilde U_\phi^*f_2\big).
\end{equation}
According to Remark~\ref{rem:two-star} this approach should not necessarily work, but if it does, it has some technical advantages.

\smallskip

Let us start by observing that by definition the algebra $(L^2(G),\star)$ is unitarily isomorphic to the algebra $\HS(L^2(V))$ of Hilbert--Schmidt operators.
Hence
\begin{equation}\label{eq:star-L2}
\|f_1\star f_2\|_2\le \|f_1\|_2\|f_2\|_2\ \ \text{for all}\ \ f_1,f_2\in L^2(G).
\end{equation}

We will need explicit formulas for the product $\star$ on
dense subspaces of  $(L^2(G),\star)$. First, we introduce an auxiliary space:
\begin{definition}
Let $\CE(G) $ be the Banach space completion of $C_c(G)$ with respect to the norm
$$
\|f\|_\CE :=\|f\|_1+\|f\|_2+\int_V\sup_{q\in Q}|f(q,v)|dv+\int_Q\sup_{v\in V}|f(q,v)|\frac{d_Q(q)}{|q|}.
$$
\end{definition}
\begin{lemma}
\label{KI}
For any $f_1,f_2\in\CE(G) $ and a.a.~$(q,v)\in G$, we have
\begin{equation}\label{eq:star}
(f_1 \star  f_2)(q,v)=
\int_G e^{i\langle q'^\flat\xi_0-\xi_0,v'\rangle}
f_1\big((q,v)(1,v')\big)\,f_2\big((q,v)(q',0)\big)\,d(q',v'),
\end{equation}
and
\begin{equation*}\label{eq:star1}
\big(\CF_V(f_1\star f_2)\big)(q,\xi)=\int_{\hat V}
(\CF_V f_1)(q,\xi')\,(\CF_V f_2)\big(\phi^{-1}(\phi(q)-\xi'),\xi-\xi'\big)\,d\xi'.
\end{equation*}
\end{lemma}

Here, following our conventions, $\CF_V\colon L^2(G)=L^2(Q\ltimes V)\to L^2(Q\times\hat V)$ is the partial Fourier transform in $V$-variables. Note that for this map to be unitary we have to equip $Q\times   \hat V$ with the measure $|q|^{-1}d_Q(q)\,d\xi$ (which in general is \emph{not} the Haar measure of the semidirect product $Q\ltimes   \hat V$ for the dual action).

\bp
For $f\in L^2(G)$, we let $K_{f}\in L^2(V\times V)$ be the  kernel of the Hilbert--Schmidt operator ${\rm Op}(f)$.
From \eqref{KO} and \eqref{Op} we get
$$
K_{f}(v,v')=\big((1\otimes\CF_V^*)\tilde U_\phi^* f\big)(v',v-v')\ \ \text{for}\ \ v,v'\in V.
$$
If $f_1,f_2\in\CE(G)$, then, since ${\rm Op}(f_1\star f_2)={\rm Op}(f_1){\rm Op}(f_2)$, the product formula for operator kernels gives
\begin{align*}
K_{f_1\star f_2}(v,v')&=\int_V K_{f_1}(v,w)K_{f_2}(w,v')\,dw\\
&=\int_V  \big((1\otimes\CF_V^*)\tilde U_\phi^* f_1\big)(w,v-w)
\big((1\otimes\CF_V^*)\tilde U_\phi^* f_2\big)(v',w-v')\,dw\\
&=\int_V  \Big(\int_{\hat V} e^{i\langle\xi_1, v-w\rangle}f_1(\phi^{-1}(\xi_1),w) \,d\xi_1\Big)
\Big(\int_{\hat V}e^{i\langle\xi_2,w-v'\rangle}f_2(\phi^{-1}(\xi_2),v')\,d\xi_2\Big)dw\\
 &=\int_{V\times\hat V\times\hat V}e^{i\langle\xi_1, v-w\rangle}e^{i\langle\xi_2,w-v'\rangle}f_1(\phi^{-1}(\xi_1),w)f_2(\phi^{-1}(\xi_2),v')\,
 dw\,d\xi_1\,d\xi_2,
\end{align*}
where the last step is justified by Fubini's theorem:  note that for any $(v,v')\in V\times V$ the function
$$
V\times\hat V\times\hat V\to \C,\quad
(w,\xi_1,\xi_2)\mapsto e^{i\langle\xi_1, v-w\rangle}e^{i\langle\xi_2,w-v'\rangle}f_1(\phi^{-1}(\xi_1),w)f_2(\phi^{-1}(\xi_2),v'),
$$
belongs to $L^1(V\times\hat V\times\hat V)$ and its $L^1$-norm is not greater than $\|f_1\|_1\|f_2\|_{\CE}$.

Next, still for $f_1,f_2\in \CE(G) $, we put
\begin{align}
\label{KJ}
f(q',v)&:=\int_G e^{i\langle q^\flat\xi_0-\xi_0,w\rangle}
f_1\big((q',v)(1,w)\big)\,f_2\big((q',v)(q,0)\big)\,d(q,w)\nonumber\\
&=\int_G e^{i\langle q^\flat\xi_0-\xi_0,w\rangle}
f_1(q',q'w+v)\,f_2(q'q,v)\,d(q,w)\nonumber\\
&=\int_G e^{i\langle q^\flat\xi_0-q'^\flat\xi_0,w-v\rangle}
f_1(q',w)\,f_2(q,v)\,d(q,w),
\end{align}
with absolutely converging integrals. It is easy to see that $f\in L^2(G)$ and
$$
\|f\|_2^2\leq\|f_1\|_1\|f_2\|_1\int_V\sup_{q\in Q}|f_1(q,v)|dv\int_Q\sup_{v\in V}|f_2(q,v)|\frac{d_Q(q)}{|q|}\leq \|f_1\|_\CE ^2 \|f_2\|_\CE^2.
$$
 In particular, we can compute $K_f$, the kernel of the operator ${\rm Op}(f)$:
\begin{align*}
K_f(v,v')&=\big((1\otimes\CF_V^*)\tilde U_\phi^* f\big)(v',v-v')
=\int_{\hat V}e^{i\langle\xi,v-v'\rangle }\,f(\phi^{-1}(\xi),v')\,d\xi\\
&=\int_{\hat V}e^{i\langle\xi,v-v'\rangle }\Big(\int_G e^{i\langle q^\flat\xi_0-\xi,w-v'\rangle}
f_1(\phi^{-1}(\xi),w)\,f_2(q,v')\,d(q,w)\Big)\,d\xi,
\end{align*}
which by Fubini (and a simplification of the phases) becomes
\begin{align*}
K_f(v,v')&=\int_{\hat V\times G}e^{i\langle\xi,v-w\rangle }e^{i\langle q^\flat\xi_0,w-v'\rangle}
f_1(\phi^{-1}(\xi),w)\,f_2(q,v')\,d(q,w)\,d\xi\\
&=\int_{\hat V\times V\times\hat V}e^{i\langle\xi,v-w\rangle }e^{i\langle \xi_2,w-v'\rangle}
f_1(\phi^{-1}(\xi),w)\,f_2(\phi^{-1}(\xi_2),v')\,d\xi\, dw\,d\xi_2.
\end{align*}
Hence $K_{f_1\star f_2}=K_f$, and~\eqref{eq:star} follows by  injectivity of the map $L^2(G)\to L^2(V\times V)$, $f\mapsto K_f$.

\smallskip

To get the second formula in the formulation of the lemma, we apply the partial Fourier transform to~\eqref{KJ}
and using Fubini's theorem one more time obtain
\begin{align*}
\big(\CF_V(f_1\star f_2)\big)(q',\xi)&=\int_G e^{i\langle \phi(q)-\phi(q'),w\rangle}
f_1(q',w)\,(\CF_V f_2)(q,\xi+ \phi(q)-\phi(q'))\,d(q,w)\nonumber\\
&=\int_Q
(\CF_V f_1)(q',-\phi(q)+\phi(q'))\,(\CF_V f_2)(q,\xi+ \phi(q)-\phi(q'))\,\frac{d_Q(q)}{|q|}\nonumber\\
&=\int_{\hat V}
(\CF_V f_1)(q',\xi')\,(\CF_V f_2)\big(\phi^{-1}(\phi(q')-\xi'),\xi-\xi'\big)\,d\xi'.
\end{align*}
This concludes the proof. \ep

\begin{corollary}
\label{Cc}
  $C_c(G)$ is a subalgebra of  $(L^2(G),\star)$.
  \end{corollary}
  \bp
Since $C_c(G)\subset \CE(G) $,  the result follows from formula \eqref{KJ} which clearly entails that when $f_1,f_2\in C_c(G)$, then also $f_1\star f_2\in C_c(G)$.
\ep

Next we consider a space of functions with good behavior in the partial Fourier space:

\begin{definition}
For a measure space $(X,\mu)$, we let $\CL(X,\mu)$ be the subspace of $L^\infty(X,\mu)$ consisting of functions that are (essentially) zero outside a set of finite measure.
We then let $\CF \CL(G)$ be the subspace of $L^2(G)$  consisting of functions of the form $\CF_V^*f$ with $f\in \CL(Q\times\hat V,|q|^{-1}d_Q(q)\,d\xi)$.
\end{definition}

\begin{lemma}
\label{KI2}
$\CF \CL(G)$ is a subalgebra of  $(L^2(G),\star)$.
Moreover, for any $f_1,f_2\in\CF\CL(G)$ and a.a.~$(q,\xi)\in Q\times\hat V$, we have
\begin{equation}\label{eq:star2}
\big(\CF_V(f_1\star f_2)\big)(q,\xi)=\int_{\hat V}
(\CF_V f_1)(q,\xi')\,(\CF_V f_2)\big(\phi^{-1}(\phi(q)-\xi'),\xi-\xi'\big)\,d\xi'.
\end{equation}
\end{lemma}

\bp
Let $K_j\subset Q$ and $\hat K_j\subset \hat V$ ($j=1,2$) be Borel sets of finite measure such that $f_j$ is
(essentially) zero outside $K_j\times \hat K_j$.
Then the function defined by the right hand side of~\eqref{eq:star2} is zero for $(q,\xi)$ outside $K_1\times (\hat K_1+\hat K_2)$. Therefore if~\eqref{eq:star2} holds, then $f_1\star f_2\in\CF\CL(G)$.

\smallskip

Turning to the proof of~\eqref{eq:star2}, we already know from Lemma~\ref{KI} that this identity holds for $f_j\in\CE(G)\cap\CF\CL(G)$. Let us write $h_n\xrightarrow[n]{\tau}f$ if $h_n\to f$ in the $L^2$-norm, $\CF_Vh_n\to\CF_Vf$ a.e., and the sequence $(\CF_Vh_n)_n$ is dominated by a function in $L^1(Q\times\hat V)\cap L^\infty(Q\times\hat V)$. Assume that we can find functions $h_{j,n}\in\CF\CL(G)$ such that $h_{j,n}\xrightarrow[n]{\tau}f_j$ and
\begin{equation}\label{eq:star2a}
\big(\CF_V(h_{1,n}\star h_{2,n})\big)(q,\xi)=\int_{\hat V}
(\CF_V h_{1,n})(q,\xi')\,(\CF_V h_{2,n})\big(\phi^{-1}(\phi(q)-\xi'),\xi-\xi'\big)\,d\xi'
\end{equation}
for almost all $(q,\xi)$. By~\eqref{eq:star-L2} we have $h_{1,n}\star h_{2,n}\to f_1\star f_2$ in $L^2(G)$ as $n\to\infty$, hence the left hand side of~\eqref{eq:star2a} (considered as a function in $(q,\xi)$) converges to $\CF_V(f_1\star f_2)$ in the $L^2$-norm. On the other hand, the right hand side converges to the right hand side of~\eqref{eq:star2} by the dominated convergence theorem. Therefore to finish the proof it suffices to show that for every $f\in\CF\CL(G)$ there exists a sequence of functions $h_n\in\CE(G)\cap\CF\CL(G)$ such that $h_n\xrightarrow[n]{\tau}f$.

First for all, if $(K_n)_n$ is an increasing sequence of compact subsets of $Q\times\hat V$ with union $Q\times\hat V$, then $\CF_V^*(f\un_{K_n})\in\CF\CL(G)$ and $\CF_V^*(f\un_{K_n})\xrightarrow[n]{\tau}f$. Therefore it suffices to consider $f\in\CF\CL(G)$ such that $\CF_Vf$ is compactly supported.

Let $K$ be any compact such that its interior contains the support of $\CF_Vf$. By Lusin's theorem, we can find a uniformly bounded sequence of continuous functions $g_n$ supported in $K$ such that $g_n\to\CF_Vf$ a.e. Then $\CF_V^*(g_n)\in\CF\CL(G)$ and $\CF_V^*(g_n)\xrightarrow[n]{\tau}f$. Therefore it suffices to consider $f$ such that $\CF_Vf\in C_c(Q\times\hat V)$.

In a similar fashion, by approximating functions in $C_c(Q\times\hat V)$ by elements of the algebraic tensor product $C_c(Q)\otimes C_c(\hat V)$, we may assume that $f=\CF_V^*(g\otimes h)=g\otimes\CF_V^*h$ for some $g\in C_c(Q)$ and $h\in C_c(\hat V)$. Finally, by approximating $h$ by the convolution of two functions, we may assume that $f=g\otimes\CF_V^*(h_1*h_2)$ for $g\in C_c(Q)$ and $h_i\in C_c(\hat V)$. But such a function is already in $\CE(G)$.
\ep

It follows from~\eqref{eq:star} that a candidate for the dual cocycle on $G$ defining the product $\star$ by formula~\eqref{eq:star-Omega} is given by
\begin{align}
\label{2CC}
\Omega&:=
\int_G e^{-i\langle q^\flat\xi_0-\xi_0,v\rangle}
\lambda_{(1,v)^{-1}}\otimes\lambda_{(q,0)^{-1}}\,d(q,v).
\end{align}
For the moment this is just a formal expression, but it at least makes sense as a sesquilinear form~$\tilde\Omega$ on $C_c(G\times G)$:
\begin{align*}
\tilde\Omega( \vf_1, \vf_2):=\int_{G} e^{-i\langle q^\flat\xi_0-\xi_0,v\rangle}
\big((\lambda_{(1,v)^{-1}}\otimes\lambda_{(q,0)^{-1}})\vf_1,\vf_2\big)\,d(q,v)\ \ \text{for}\ \ \vf_1,\vf_2\in C_c(G\times G).
\end{align*}

Our first goal is to prove that $\Omega$ makes sense as a unitary operator
on $L^2(G\times G)$. This will be proven by showing that $\Omega$ factorizes as a product of three unitaries.

Consider the following almost everywhere defined measurable transformation:
\begin{align*}
\Xi:Q\times  \hat V \times G\to Q\times  \hat V \times G,\quad
(q,\xi,g)
 \mapsto
\big(q,\xi,\phi^{-1}(\xi_0+\xi)g\big).
\end{align*}
The operator $U_\Xi$ on $L^2\big(Q\times   \hat V\times G,|q|^{-1}d_Q(q)\,d\xi\,dg \big)$ mapping $f$ into $f\circ\Xi$
is unitary. We then have:

\begin{lemma}
\label{lem2}
The convolution operator $\Omega$ factorizes as follows:
$$
\Omega=(\CF_ {V}^*\otimes 1) \,U_\Xi\,(\CF_V\otimes 1).
$$
\end{lemma}

\begin{proof}
For  $\vf_1,\vf_2\in C_c(G\times G)$, the function
$$
(q,v;g_1,g_2)\mapsto
\vf_1\big((1,v)g_1;(q,0)g_2 \big)\overline{\vf_2\big(g_1,g_2\big)} ,
$$
belongs to $L^1(G^3)$. Hence, we may use  Fubini's theorem to write $\tilde\Omega( \vf_1, \vf_2)$
as follows:
\begin{align*}
& \int_{G^3} e^{-i\langle q^\flat\xi_0-\xi_0,v\rangle}
\vf_1\big((1,v)(q_1,v_1);(q,0)(q_2,v_2) \big)\,\overline{\vf_2\big(q_1,v_1;q_2,v_2\big)}\,d(q,v)\,d(q_1,v_1)\,d(q_2,v_2) \\
&= \int_{G^3} e^{-i\langle q^\flat\xi_0-\xi_0,v\rangle}
\vf_1\big(q_1,v+v_1;qq_2,qv_2\big)\overline{\vf_2\big(q_1,v_1;q_2,v_2\big)} \,d(q,v)\,d(q_1,v_1)\,d(q_2,v_2)\\
&= \int_{G^3} e^{-i\langle q^\flat\xi_0-\xi_0,v-v_1\rangle}
\vf_1\big(q_1,v;qq_2,qv_2\big)\overline{\vf_2\big(q_1,v_1;q_2,v_2\big)} \,d(q,v)\,d(q_1,v_1)\,d(q_2,v_2)\\
&= \int_{Q^3\times V}
 \,\big((\CF_V\otimes 1)\vf_1\big)\big(q_1,q^\flat\xi_0-\xi_0;qq_2,qv_2\big) \overline{\big((\CF_V\otimes 1)\vf_2\big)\big(q_1,q^\flat\xi_0-\xi_0;q_2,v_2\big)}\\
&\hspace{10cm}\times\frac{d_Q(q_1)}{|  q_1|}\frac{d_Q(q_2)\,dv_2}{|q_2|}\frac{d_Q(q)}{|q|}\\
&=\int_{Q\times\hat V\times Q\times V}
\big((\CF_V\otimes 1)\vf_1\big)\big(q_1,\xi;\phi^{-1}(\xi_0+\xi)q_2,\phi^{-1}(\xi_0+\xi)v\big)
\overline{\big((\CF_V\otimes 1)\vf_2\big)\big(q_1,\xi;q_2,v\big)}\\
&\hspace{10cm}\times
\frac{d_Q(q_1)\,d\xi}{|  q_1|}\,\frac{d_Q(q_2)\,dv}{|  q_2|},
\end{align*}
which completes the proof.
\end{proof}

Next, by the definition of $\Omega$ it is clear that $\Omega$ commutes with the operators $\rho_g\otimes1$ and $1\otimes\rho_g$. Hence $\Omega\in W^*(G)\bar\otimes W^*(G)$.

\begin{lemma}
For all $f_1,f_2\in A(G)\cap L^2(G)$, we have
$$
f_1\star f_2=(f_1\otimes f_2)(\Dhat(\cdot)\Omega^*).
$$
\end{lemma}

\bp
Recall that $A(G)$ consists of functions of the form $\vf_1*\check \vf_2$, with $\vf_i\in L^2(G)$, which correspond to the linear functionals
$(\cdot\,\vf_2,\bar\vf_1)$  on $W^*(G)$. Under the identification $A(G)\simeq W^*(G)_*$,  Lemma \ref{KI} says that if  $f_1,f_2\in \CE(G) \cap A(G)$, then
$$
(f_1\star f_2)(g)=
 \int_G e^{i\langle q'^\flat\xi_0-\xi_0,v'\rangle}
(f_1\otimes f_2)\big(\lambda_{g(1,v')}\otimes\lambda_{g(q',0)}\big)\,d(q',v').
$$
 Now, for $f_j=\vf_j\ast\check\vf_j'$, with $\vf_j,\vf_j'\in C_c(G)$, we have
$$
 (f_1\otimes f_2)(\Dhat(g)\Omega^*)=(\Dhat(g)\Omega^*(\vf_1'\otimes\vf_2'),\bar\vf_1\otimes\bar\vf_2).
$$
Using the initial definition of $\Omega$ as a bilinear form on $C_c(G\times G)$, we get
\begin{align*}
 (f_1\otimes f_2)(\Dhat(g)\Omega^*)&=  \int_G e^{i\langle q'^\flat\xi_0-\xi_0,v'\rangle}
\big((\lambda_{g(1,v')}\otimes\lambda_{g(q',0)})(\vf_1'\otimes\vf_2'),\bar\vf_1\otimes\bar\vf_2\big)
\,d(q',v')\\
&=
 \int_G e^{i\langle q'^\flat\xi_0-\xi_0,v'\rangle}
(f_1\otimes f_2)\big(\lambda_{g(1,v')}\otimes\lambda_{g(q',0)}\big)\,d(q',v').
\end{align*}
Hence the equality in the formulation of the lemma holds for all $f_1,f_2$ of the form $\vf_1*\check \vf_2$, with $\vf_i\in C_c(G)$.
Therefore in order to prove the lemma it suffices to show that every function $f\in A(G)\cap L^2(G)$ can be approximated  by functions of the form
$\vf_1*\check \vf_2$, with $\vf_i\in C_c(G)$, simultaneously in the norms on $A(G)$ and $L^2(G)$.

Consider first a function of the form $f=f_1*\check f_2$, with $f_1\in L^2(G)$ and $f_2\in C_c(G)$. If $\varphi_n\to f_1$ in $L^2(G)$, $\vf_n\in C_c(G)$, then $\vf_n*\check f_2\to f_1*\check f_2$ both in $A(G)$ and $L^2(G)$. Consider now an arbitrary $f\in A(G)\cap L^2(G)$. By the previous case in order to finish the proof it suffices to show that $f$ can be approximated simultaneously in $A(G)$ and $L^2(G)$ by functions of the form $f*\check\vf$, with $\vf\in C_c(G)$. Take a standard approximate unit $(\vf_n)_n$ in $L^1(G)$ consisting of functions $\vf_n\in C_c(G)$ such that $\vf_n\ge0$, $\int_G\vf_n\,dg=1$, with the supports of $\vf_n$ eventually contained in arbitrarily small neighbourhoods of the unit. Then $f*\check\vf_n\to f$ in $L^2(G)$. At the same time, if we write $f$ as $f_1*\check f_2$ for some $f_i\in L^2(G)$ and use that
$$
f*\check\vf_n=f_1*(\vf_n*f_2)\check{}
$$
and $\vf_n*f_2\to f_2$ in $L^2(G)$, we also see that $f*\check\vf_n\to f$ in $A(G)$.
\ep

We thus see that $A(G)\cap L^2(G)$ is a subalgebra of $(L^2(G),\star)$. Since $A(G)\cap L^2(G)$ is dense in~$A(G)$, the associativity of the product $\star$ on this subalgebra implies that $\Omega$ satisfies the cocycle identity~\eqref{eq:dual-cocycle}.

To summarize, we have proved the following result.

\begin{theorem}\label{thm:KN1}
For any second countable locally compact group $G=Q\ltimes V$ satisfying the dual orbit Assumption~\ref{Fro},
formula~\eqref{2CC} defines a dual unitary $2$-cocycle $\Omega$ on $G$. The corresponding product $\star_\Omega$ on $A(G)$ coincides on $A(G)\cap L^2(G)$ with the product~$\star$ defined by~\eqref{eq:star0}.
\end{theorem}

\begin{remark}
\label{BO}
If we started from the opposite Kohn--Nirenberg quantization (see Remark \ref{AKNQ})
we would have obtained the following dual $2$-cocycle:
\begin{equation}\label{eq:anti-Omega}
\overline\Omega:=\int_G e^{i\langle q^\flat\xi_0-\xi_0,v\rangle}
\lambda_{(q,0)^{-1}}\otimes\lambda_{(1,v)^{-1}}\,d(q,v),
\end{equation}
which differs from $\Omega$ by the inversion of legs and by the sign of the phase:
$$
\overline\Omega=(J\otimes J)\Omega_{21}(J\otimes J)=(\hat R\otimes\hat R)(\Omega^*_{21}),
$$
where $\hat R$ is the unitary antipode of $W^*(G)$ given on the generators by $\hat R(\lambda_g)=\lambda_{g^{-1}}$.
By~\cite[Proposition 6.3, iii)]{DC}, the dual cocycles $\overline \Omega$ and $\Omega$ are cohomologous,
with the unitary operator implementing the cohomological relation equal to $\tilde JJ$, which we will explicitly compute in Section~\ref{ss:mult}.
\end{remark}

\subsection{Identification of the Galois objects}

To complete the picture it remains to check that the Galois object defined by the dual cocycle $\Omega$ is exactly the pair $(B(L^2(V)),\Ad\pi)$.

For $f\in L^2(G)$, consider the operator $L^\star(f)$  on $L^2(G)$ defined by
$$
L^{\star }(f)f'=f\star  f'.
$$
By~\eqref{eq:star-L2} we have
$$
\|L^\star(f)\|\le\|f\|_2.
$$
Furthermore, under the identification of $(L^2(G),\star)$ with $\HS(L^2(V))$ via $\Op$, the map $L^\star$ is simply the left regular representation of $\HS(L^2(V))$ on itself. This representation is a multiple of the canonical representation of $\HS(L^2(V))$ on $L^2(V)$. It follows that there is a unique isomorphism
\begin{equation}\label{eq:Galois-star}
B(L^2(V))\cong L^\star(L^2(G))''\ \ \text{such that}\ \ \Op(f)\mapsto L^\star(f).
\end{equation}
Therefore in order to find an isomorphism $W^*(\hat G;\Omega)\cong B(L^2(V))$ it suffices to find an isomorphism $W^*(\hat G;\Omega)\cong L^\star(L^2(G))''$.

From formula~\eqref{eq:GNS0} for the GNS-map on $W^*(\hat G;\Omega)$, for every $f\in A(G)$ we have the following equality:
$$
\pi_\Omega(f)\vf= S\,L^{\star }(f)\,S\vf
$$
for all $\vf\in A(G)\cap L^2(G)$ such that the right hand side is well-defined,
where $S$ is the unbounded operator defined by $S\vf=\check\vf$. In other words, using the unitary operator $\CJ$ defined in~\eqref{J},
$$
\mathcal J=J\hat J=\hat J J=M(\Delta^{-1/2}) S=SM(\Delta^{1/2}),
$$
we have
\begin{equation}
\label{F1}
\pi_\Omega(f)\vf= \mathcal J\,M(\Delta^{-1/2})\,L^{\star }(f)\,M(\Delta^{1/2})\,\mathcal J\vf
\end{equation}
for all $\vf\in A(G)\cap L^2(G)$ such that the right hand side is well-defined.

We thus see that we need to understand a connection between the operators $L^\star(f)$ and $M(\Delta^s)$. For this we will first get another useful formula for $L^\star$.

First,
we denote by $U_\phi$ the variant of the unitary operator $\tilde U_\phi $, defined in~\eqref{eq:tilde-U-phi}, without the permutations of variables:
\begin{equation}
 \label{UPHI}
 U_\phi :L^2(\hat V\times V)\to L^2(G),\quad
(U_\phi f)(q,v):=f\big(\phi(q),v \big).
\end{equation}
Then we denote by $\gamma$ the unitary representation of $\hat V$ on $L^2(G)$ given by
\begin{align}
\label{pi}
\gamma(\xi)=U_\phi \CF_V^*\,(\tau_{\xi}\otimes \tau _{\xi}) \CF_{V}U_\phi^*,
\end{align}
where $\tau_\xi:L^2(\hat V)\to L^2(\hat V)$ is the left regular representation of $\hat V$ given by $(\tau_\xi\vf)(\xi')=\vf(\xi'-\xi)$.

Next, for fixed $\xi\in \hat V$ and $f\in\CF\CL(G)$, we denote by $(\CF_{V} f)(\bullet,\xi)$
the $V$-invariant function on~$G$ given by  $\big[(q,v)\mapsto (\CF_{ V} f)(q,\xi)\big]$
and by $M\big((\CF_{V}f) (\bullet,\xi)\big)$ the bounded operator
on $L^2(G)$ of  multiplication by the function $(\CF_{V}f) (\bullet,\xi)$.

Lemma~\ref{KI2} then leads to the following result.

\begin{lemma}
\label{B2}
For any $f\in  \CF\CL(G)$, we have the absolutely convergent (in the operator norm) integral formula
\begin{equation}
\label{precious}
L^{\star }(f)=\int_{\hat V}M\big((\CF_{ V} f)(\bullet,\xi)\big)\,
\gamma(\xi)\,d\xi,
\end{equation}
\end{lemma}

Finally, we introduce a family $(T_z)_{z\in\C}$ of operators on functions on $G$ by
\begin{equation}\label{eq:T}
(T_zf)(q,v)=\int_{\hat V} \Delta(\phi^{-1}(\xi_0-{q^{-1}}^\flat \xi))^{-z}\,e^{i\langle \xi,v\rangle}\,
(\CF_Vf)(q,\xi)\,d\xi,
\end{equation}
where we remind that $\Delta(q')=\Delta_Q(q')/|q'|$.

We need a dense subspace of $\CF\CL(G)$ preserved by these operators:

\begin{definition}
For compact subsets $K,L\subset Q$,  let $\CL_{K,L}(Q\times\hat V)$ be the subspace of $\CL(Q\times \hat V)$ consisting of functions
supported on the compact set
\begin{equation}
\label{eq:suppKL}
\big\{(q,\xi)\in Q\times\hat V\mid q\in K,\ \xi_0-{q^{-1}}^\flat\xi\in \phi(L) \big\}.
\end{equation}
We denote by $\CL_0(Q\times\hat V)$  the union of the spaces $\CL_{K,L}(Q\times\hat V)$ and
by $\CF\CL_0(G)$ (respectively, by $\CF\CL_{K,L}(G)$) the subspace of $\CF\CL(G)$ consisting of functions $f\in\CF\CL(G)$ such that $\CF_Vf$ belongs to $\CL_0(Q\times\hat V)$ (respectively, to $\CL_{K,L}(Q\times\hat V)$).
\end{definition}

\begin{lemma} \label{lem:fl-dense}
The space $\CF\CL_0(G)$ is dense  in $L^2(G)$ and  $A(G)\cap\CF\CL_0(G)$ is dense in $A(G)$. Moreover,  $\CF\CL_0(G)$ is stable under $T_z$  and  $T_z(A(G)\cap\CF\CL_0(G))$ is dense in $L^2(G)$ for any $z\in\C$.
\end{lemma}

\bp
By definition~\eqref{eq:T}, the operator $T_z$ conjugated by the partial Fourier transform
is the operator of multiplication by the function
$$
(q,\xi)\mapsto  \Delta\big(\phi^{-1}(\xi_0-{q^{-1}}^\flat \xi)\big)^{-z}.
$$
This immediately shows that  $\CF\CL_{K,L}(G)$ is stable under $T_z$  as the modular function $\Delta$ is bounded, as well as bounded away from zero, on any compact subset
of $Q$. Hence $\CF\CL_0(G)$ is also stable under $T_z$.

\smallskip

Since $\F_V$ is unitary, to prove density of $\CF\CL_0(G)$ in $L^2(G)$ it suffices to show that $\CL_0(Q\times\hat V)$ is dense in $L^2(Q\times\hat V,|q|^{-1}d_Q(q)\,d\xi)$. For this it suffices to show that for every compact set $K\subset Q$ the union of the sets~\eqref{eq:suppKL} over the compact sets $L\subset Q$ is a subset of $K\times\hat V$ of full measure. But this is clear, since this union is
$$
\big\{(q,\xi)\in Q\times\hat V\mid q\in K,\ \xi\in q^\flat\xi_0-\phi(Q)\big\}
$$
and by assumption $\phi(Q)$ is a subset  of $\hat V$ of full measure.

\smallskip

Since the map $f\mapsto \check f$ is bounded on $L^2(K\times V,|q|^{-1}d_Q(q)\,d\xi)\subset L^2(G)$, with image $L^2(K^{-1}\times V,|q|^{-1}d_Q(q)\,d\xi)$, the functions $\check f$ for $f\in\CF\CL_0(G)$
form a dense subspace of $L^2(G)$ as well. Hence the functions
$\vf* f$ for $\vf\in C_c(G)$ and $f\in\CF\CL_0(G)$ are dense in $A(G)$.
An easy calculation shows that for any $g=(q,v)\in G$ and compacts $K,L\subset Q$, we have $\lambda_g(\CF\CL_{K,L}(G))\subset\CF\CL_{qK,L}(G)$.
Hence, if $f\in \CF\CL_{K,L}(G)$ and $\vf\in C_c(G)$ has support contained in $U\times V$ for a compact set $U\subset Q$, then
$\vf* f\in\CF\CL_{UK,L}(G)$.
Therefore $\vf*f\in A(G)\cap\CF\CL_0(G)$ for all $\vf\in C_c(G)$ and $f\in\CF\CL_0(G)$. Thus $A(G)\cap\CF\CL_0(G)$ is dense in $A(G)$.

Taking a standard approximate unit in $L^1(G)$ for $\varphi$, we see also that  functions of the form $\vf* f$, for $\vf\in C_c(G)$ and $f\in\CF\CL_0(G)$, are dense in
$\CF\CL_0(G)$, hence
in $L^2(G)$. Moreover, since the operators $T_z$ are bounded on the spaces $\CF\CL_{K,L}(G)$, we may also conclude that the functions $T_z(\vf*f)$ are dense in $L^2(G)$
 for all $z$. In particular, $T_z(A(G)\cap\CF\CL_0(G))$ is dense in $L^2(G)$ for all $z$.
\ep

\begin{proposition}\label{prop:chi-star}
The operator $M(\Delta)$ is affiliated with the von Neumann algebra $L^\star(L^2(G))''$.
Moreover,  for all $f\in\CF\CL_0(G)$ and $z\in\C$, we have
$$
M(\Delta^z)L^\star(f)=L^\star(\Delta^zf)\ \ \text{on}\ \ L^2(G),\ \ M(\Delta^z)L^\star(f)M(\Delta^{-z})=L^\star(T_zf)
\ \ \text{on}\ \ C_c(G).
$$
\end{proposition}

\bp
Since $\Delta$ depends only on the coordinate $Q$, the operators $M(\Delta^{it})$, $t\in\R$, commute with the partial Fourier transform $\CF_V$.
From formula~\eqref{eq:star2} we see then that
$$
M(\Delta^{it})L^\star(f)=L^\star(\Delta^{it}f)\ \ \text{for all}\ \ f\in\CF\CL(G).
$$
As $L^\star(\CF\CL_0(G))$ is $\sigma$-weakly dense in $L^\star(L^2(G))''$, this implies that $M(\Delta)$ is affiliated with $L^\star(L^2(G))''$. The same formula also shows that
since $\Delta^z f\in\CF\CL_0(G)$ for $f\in \CF\CL_0(G)$, we have
$M(\Delta^z)L^\star(f)=L^\star(\Delta^zf)$ on $L^2(G)$.

\smallskip

Next, formula \eqref{precious} and definition~\eqref{eq:T} of $T_z$ give, for $f\in \CF\CL_0(G)$, the absolutely convergent integral
\begin{align}
\label{T12}
L^{\star }(T_{z}f)=\int_{\hat V}
M\Big(\Delta\big(\phi^{-1}(\xi_0-{\bullet^{-1}}^\flat \xi)\big)^{-z}
(\CF_{ V} f)(\bullet,\xi)\Big)\,
\gamma(\xi)\,d\xi.
\end{align}

On the other hand, using again \eqref{precious}, we have on $\Delta^{\Re(z)}L^2(G)$:
$$
L^{\star }(f)M(\Delta^{-z})=\int_{\hat V}M\big((\CF_{ V} f)(\bullet,\xi)\big)\,
\gamma(\xi)M(\Delta^{-z})\,d\xi.
$$
Now, for $\vf\in \Delta^{\Re(z)}L^2(G)$, we have
\begin{align*}
(\gamma(\xi)M(\Delta^{-z})\vf)(q,v)&=\big(U_\phi \CF_{V}^*\,(\tau_{\xi}\otimes \tau _{\xi}) \CF_{V}U_\phi^*(\Delta^{-z}\vf)\big)(q,v).
\end{align*}
As $\CF_V$ commutes with operators of multiplication by $V$-invariant functions, from this we easily deduce:
 \begin{align*}
 \big(\gamma(\xi)M(\Delta^{-z})\vf\big)(q,v)&=\Delta(\phi^{-1}(q^\flat\xi_0-\xi))^{-z}(\gamma(\xi)\vf)(q,v)\\
  &=\Delta(q)^{-z}\Delta\big(\phi^{-1}(\xi_0-{q^{-1}}^\flat\xi))^{-z}(\gamma(\xi)\vf)(q,v).
 \end{align*}
Hence we have
 $$
 \gamma(\xi)M(\Delta^{-z})=M(\Delta^{-z})M\big( \Delta(\phi^{-1}(\xi_0-{\bullet^{-1}}^\flat\xi)^{-z})\big)\, \gamma(\xi),
 $$
 which by closedness of $M(\Delta^{-z})$ finally gives
 \begin{align*}
 L^{\star }(f)M(\Delta^{-z})=M(\Delta^{-z}) \int_{\hat V}
M\Big( \Delta\big(\phi^{-1}(\xi_0-{\bullet^{-1}}^\flat \xi\big)^{-z}\Big)
(\CF_{ V} f)(\bullet,\xi)\bigg)\,\gamma(\xi)\,d\xi
 \end{align*}
on $L^2(G)\cap\Delta^{\Re(z)}L^2(G)$.
Together with~\eqref{T12} this shows that the identity
$$
L^\star(f)M(\Delta^{-z})=M(\Delta^{-z})L^\star(T_zf)
$$
holds on $L^2(G)\cap\Delta^{\Re(z)}L^2(G)$, in particular, on $C_c(G)$.
\ep

This proposition and identity~\eqref{F1} imply that for any $f\in A(G)\cap\CF\CL_0(G)$ we have
\begin{equation}\label{eq:Omega-L}
\pi_\Omega(f)=\CJ L^\star(T_{-1/2}f)\CJ
\end{equation}
on $C_c(G)*C_c(G)\subset A(G)\cap C_c(G)$, hence on $L^2(G)$, as both sides of the identity are bounded operators. Since $A(G)\cap\CF\CL_0(G)$ is dense in $A(G)$ and $T_{-1/2}(A(G)\cap\CF\CL_0(G))$ is dense in $L^2(G)$ by Lemma~\ref{lem:fl-dense}, it follows that $W^*(\hat G;\Omega)=\CJ L^\star(L^2(G))''\CJ$. Recalling also that the action of $G$ on $W^*(\hat G;\Omega)$ is given by the automorphisms $\Ad\rho_g$, we see that this action transforms under the isomorphism $\Ad\CJ$ of $W^*(\hat G;\Omega)$ onto $L^\star(L^2(G))''$ into the action given by the automorphisms $\Ad\lambda_g$. Using the isomorphism~\eqref{eq:Galois-star} the latter action transforms, in turn, into the action $\Ad\pi$ on $B(L^2(V))$. We therefore get the required isomorphism of the Galois objects:

\begin{theorem}\label{thm:KN2}
For any second countable locally compact group $G=Q\ltimes V$ satisfying the dual orbit Assumption~\ref{Fro}, the $G$-Galois object $(W^*(\hat G;\Omega),\beta)$ defined by the dual cocycle~\eqref{2CC} is isomorphic to $(B(L^2(V)),\Ad\pi)$; explicitly, the isomorphism maps $\pi_\Omega(f)$, $f\in A(G)\cap\CF\CL_0(G)$, into $\Op(T_{-1/2}f)$.
\end{theorem}

As a byproduct we see that $(B(L^2(V)),\Ad\pi)$ is indeed a Galois object. We remind that by Theorem~\ref{thm:genuine-rep} this is therefore the unique up to isomorphism I-factorial Galois object defined by a genuine representation of~$G$.

\begin{remark}
Formula~\eqref{precious} shows that it is natural to extend the representation $L^\star$ to a larger class of functions including all measurable functions on $Q$ (viewed as functions on $G$) by letting $L^\star(f)=M(f)$ for such functions. With this definition we have $\Delta^s\star\Delta^t=\Delta^{s+t}$ and, by Proposition~\ref{prop:chi-star}, $\Delta^s\star f=\Delta^sf$, $f\star\Delta^s=\Delta^s T_{-s}f$ for $f\in\CF\CL_0(G)$. We are therefore exactly in the situation discussed in Remark~\ref{rem:two-star}, with identities~\eqref{eq:chi-star} satisfied for $c=1$ and $\alpha=1/2$. This ``explains'' why we were able to construct the dual cocycle $\Omega$ using the quantization map~$\Op$ instead of the modified quantization map $\Op'$.
\end{remark}

\subsection{Deformation of the trivial cocycle}\label{ss:deform}

We continue to consider a second countable locally compact group $G=Q\ltimes V$ satisfying the dual orbit Assumption~\ref{Fro} and a dual unitary $2$-cocycle $\Omega$ on $G$ defined by~\eqref{2CC}.

By replacing the distinguished point $\xi_0\in\hat V$ by $\xi^q_0:=q^\flat\xi_0$ we in fact get a family of such dual cocycles $\Omega_q$ indexed by the elements $q\in Q$. We already know that all these cocycles are cohomologous, since they correspond to the unique Galois object defined by a genuine representation. This is also easy to see as follows.

\begin{proposition}\label{prop:cont}
We have $\Omega_q=(\lambda_q\otimes\lambda_q)\Omega\Dhat(\lambda_q)^*$ for all $q\in Q$.
In particular, the map $q\mapsto\Omega_q$ is continuous in the so-topology.
\end{proposition}

\bp
Considering as before the partial Fourier transform $\CF_V$ as a map $L^2(G)\to L^2(Q\times\hat V,|q|^{-1}d_Q(q)\,d\xi)$, we have
$$
(\CF_V\lambda_q\CF_V^*f)(q',\xi')=|q|f(q^{-1}q',{q^{-1}}^\flat\xi').
$$
From this and Lemma~\ref{lem2} we get
$$
(\lambda_q\otimes\lambda_q)\Omega\Dhat(\lambda_q)^*=(\CF_V^*\otimes1)U_{\Xi_q}(\CF_V\otimes1),
$$
where $\Xi_q\colon Q\times\hat V\times G\to Q\times\hat V\times G$ is defined by
$$
\Xi_q(q',\xi',g):=(q',\xi',q\phi^{-1}(\xi_0+{q^{-1}}^\flat\xi')q^{-1}g).
$$

On the other hand, consider the map $\phi_q\colon Q\to\hat V$ defined by $\xi^q_0=q^\flat\xi_0$, so that
$$
\phi_q(q'):=q'\xi^q_0=\phi(q'q).
$$
Then
$$
q\phi^{-1}(\xi_0+{q^{-1}}^\flat\xi')q^{-1}=\phi^{-1}(\xi_0^q+\xi')q^{-1}=\phi_q^{-1}(\xi_0^q+\xi'),
$$
which shows that $\Xi_q$ is exactly the map from the expression for $\Omega_q$ in Lemma~\ref{lem2} (with $\xi_0$ replaced by $\xi^q_0$). This proves the proposition.
\ep

Although the dual cocycles $\Omega_q$ are all cohomologous, under quite general assumptions they can be used to construct a continuous deformation of the trivial cocycle. Namely, we have the following result.

\begin{proposition}\label{prop:1-deform}
Assume that $G=Q\ltimes V$ in addition to satisfying the dual orbit Assumption~\ref{Fro} is such that the map $\phi\colon Q\to\hat V$ is open. Assume also that there exists a sequence $(z_n)_n$ of elements of the center $Z(Q)$ of $Q$ such that $z_n^\flat\to0$ in $\End(\hat V)$ pointwise. Then $\Omega_{z_n^{-1}}\to1$ in the so-topology.
\end{proposition}

\bp
Using the notation from the proof of the previous proposition it suffices to show that $\Xi_{z_n^{-1}}\to \id$ a.e., or equivalently,
$\phi^{-1}(\xi_0+z_n^\flat\xi)\to e$ for a.e.~$\xi\in\hat V$.  But this is clear, since by assumption the image of $\phi$ contains a neighbourhood of $\xi_0$ and $\phi$ is a homeomorphism of $Q$ onto its image.
\ep

\begin{example}\label{ex:ax+b}
Consider the $ax+b$ group $G$ over the reals, so that $Q=\R^*$, $V=\R$, and $Q$ acts on $V$ by multiplication. In other words, $G$ is the group of matrices
$$
\left\{\begin{pmatrix}a & b\\ 0 & 1 \end{pmatrix}\mid a\in\R^*,\ b\in\R\right\}.
$$
We identify $\hat \R$ with $\R$ via the pairing $e^{ixy}$. Then $s^\flat t=s^{-1}t$ for $s\in Q$ and $t\in\hat V$.
Take $-1$ as~$\xi_0$. We then get a continuous family of cohomologous dual unitary $2$-cocycles $\Omega_\theta$, $\theta\in\R^*$, on~$G$ such that
$$
\Omega_\theta=(\CF_\R^*\otimes1)U_{\Xi_\theta}(\CF_\R\otimes1),
$$
where $\Xi_\theta\colon \R^*\times \R\times G\to \R^*\times\R\times G$ is defined by
$$
\Xi_\theta(s,t,g):=\big(s,t,\begin{pmatrix}(1-\theta t)^{-1} & 0\\ 0 & 1 \end{pmatrix}g\big).
$$
In this case the pointwise convergence $\theta^\flat\to0$ in $\End(\hat V)$ means that $\theta^{-1}\to0$ in $\R$. So by the above proposition we have $\Omega_\theta\to 1$ as $\theta\to0$, which is obviously the case.
\end{example}

\subsection{Multiplicative unitaries}\label{ss:mult}

Our next goal is to find an explicit formula for the multiplicative unitary of the twisted quantum group $(W^*(G),\Omega\Dhat(\cdot)\Omega^*)$ for the dual cocycle $\Omega$ defined by~\eqref{2CC}. The main issue is to determine the modular conjugation $\tilde J$ for the canonical weight~$\tilde\varphi$ on $W^*(\hat G;\Omega)$.

\smallskip

It is more convenient to work with the isomorphic Galois object $(L^\star(L^2(G))'',\Ad\lambda)$. The algebra $(L^2(G),\star)$ becomes a $*$-algebra if we transport the $*$-structure on $\HS(L^2(V))$ to $L^2(G)$. Since the $*$-structure on the algebra of Hilbert--Schmidt operators is isometric, the corresponding $*$-structure on $L^2(G)$ must have the form $f\mapsto\U Jf$ for a unitary operator $\U$ on $L^2(G)$. In other words, the operator $\U$ is defined by the identity
$$
\Op(f)^*=\Op(\U Jf).
$$

\begin{lemma}\label{lem:U}
The operator $\U$ is given by
$$
\U=U_\phi\CF_{V}^*\hat W_{\hat V}^*\CF_VU_\phi^*,
$$
where $U_\phi$ is the operator defined by~\eqref{UPHI}.
\end{lemma}

\bp
Consider the unitary $\widetilde{\Op}\colon L^2(V\times V)\to \HS(L^2(V))$ defined by
$$
\widetilde{\Op}(f)=\OKN((1\otimes\CF_V)f).
$$
Then by definition $\widetilde{\Op}(f)$ is the integral operator with kernel $(v,v')\mapsto f(v',v-v')$, hence its adjoint has the kernel $(v,v')\mapsto\overline{f(v,v'-v)}$. In other words,
$$
\widetilde{\Op}(f)^*=\widetilde{\Op}(f^\#),
$$
where $f^\#(v,v')=\overline{f(v+v',-v')}$. Since
$$
f^\#=(J_V\otimes\hat J_V)\hat W_V^*f,
$$
we conclude that
$$
\U J=\tilde U_\phi(1\otimes\CF_V)(J_V\otimes\hat J_V)\hat W_V^*(1\otimes\CF_V^*)\tilde U_\phi^*.
$$
Using that
$$
(J_V\otimes \hat J_V)\hat W_V^*=\hat W_V (J_V\otimes \hat J_V) \quad\mbox{and}\quad \hat J_V \CF_V^*=  \CF_V^*J_{\hat V},
$$
we arrive at
$$
\U J=\tilde U_\phi(1\otimes\CF_V)\hat W_V(1\otimes\CF_V^*)\tilde U_\phi^*J.
$$
Finally, using that $U_\phi=\tilde U_\phi\Sigma$ and
$$
\hat W_V=\Sigma(\CF_V^*\otimes\CF_V^*)\hat W^*_{\hat V}(\CF_V\otimes\CF_V)\Sigma,
$$
we get the desired formula.
\ep

\begin{proposition}
The modular conjugation for the canonical weight $\tilde\varphi$ on the Galois object $W^*(\hat G;\Omega)$ (with respect to the GNS-map~\eqref{eq:GNS0}) is $\tilde J=\CJ\U J\CJ$.
\end{proposition}

\bp The proof is similar to that of~\cite[Proposition~2.8]{BGNT1}.
Let us start with the canonical weight~$\tilde\varphi_L$ for the Galois object $(L^\star(L^2(G))'',\Ad\lambda)$. Note that since $M(\Delta)$ is affiliated with $L^\star(L^2(G))''$ and the function $\Delta$ is, up to a scalar factor, the only positive measurable function~$F$ on~$G$ such that $\lambda_g F=\Delta(g)^{-1}F$, the isomorphism $L^\star(L^2(G))''\cong B(L^2(V))$, $L^\star(f)\mapsto\Op(f)$, must map $M(\Delta^{it})$ ($t\in\R$) to $c^{it} K^{-it}$ for some $c>0$, where $K$ is the Duflo--Moore operator of formal degree (explicitly given by Lemma~\ref{lem:DM}). We have densely defined operators on $\HS(L^2(V))$ of multiplication on the right by $c^zK^{-z}$ ($z\in\C$). Correspondingly, we have densely defined operators on $L^2(G)$, which we suggestively denote by $f\mapsto f\star\Delta^z$. Thus, by definition,
$$
L^\star(f\star\Delta^z)=L^\star(f)M(\Delta^z)
$$
for $f$ in a dense subspace of $L^2(G)$. Explicitly, by Proposition~\ref{prop:chi-star} we have
$$
f\star\Delta^z=\Delta^z T_{-z}f\ \ \text{for}\ \ f\in\CF\CL_0(G).
$$

Recalling the description of the GNS-representation for $(B(L^2(V)),\Ad\pi)$ in Section~\ref{ss:dual-cocycles}, and formula~\eqref{eq:GNS} in particular, we see that as the GNS-space for $\tilde\varphi_L$ we can take $L^2(G)$, with the GNS-map $\tilde\Lambda_L\colon{\mathfrak N}_{\tilde\varphi_L}\to L^2(G)$ uniquely determined by
$$
\tilde\Lambda_L(L^\star(f))=c^{1/2}f\star\Delta^{-1/2}
$$
for $f\in L^2(G)$ such that the right hand side is well-defined. The corresponding modular conjugation $\tilde J_L$ is simply given by the involution on $L^2(G)\cong\HS(L^2(V))$, so $\tilde J_L=\U J$.

Now, using the isomorphism $\Ad\mathcal J$ between $W^*(\hat G;\Omega)$ and $L^{\star }(L^2(G))''$, we can consider the space $L^2(G)$ as the GNS-space for $\tilde\varphi$ using the map
\begin{equation}\label{eq:tildelambda2}
\mathfrak N_{\tilde\varphi}\ni x\mapsto \mathcal J\tilde\Lambda_L(\mathcal Jx\mathcal J).
\end{equation}
In this picture the modular conjugation for $\tilde\varphi$ is $\mathcal J\mathcal UJ \mathcal J$. Therefore we only have to check that the above GNS-map is exactly the map $\tilde\Lambda$ used to define $\tilde J$.

Recall that $\tilde\Lambda$ is given by
$$
\tilde\Lambda(\pi_\Omega(f))=\check{f}=\mathcal J\Delta^{-1/2}f
$$
for all $f\in A(G)\cap \Delta^{1/2}L^2(G)$.
Since by~\eqref{eq:Omega-L} we have
$$
\mathcal J\pi_\Omega(f)\mathcal J=L^{\star }(T_{-1/2} f)=L^\star((\Delta^{-1/2}f)\star\Delta^{1/2})
$$
for $f\in A(G)\cap\CF\CL_0(G)$, we see that
$$
\CJ\tilde\Lambda_L(\CJ\pi_\Omega(f)\CJ)=c^{1/2}\tilde\Lambda(\pi_\Omega(f))
$$
for all such $f$. By comparing the norms of both sides we can already conclude that $c=1$. Then, since any two GNS-representations associated with $\tilde\varphi$ are unitary conjugate and the vectors $\tilde\Lambda(\pi_\Omega(f))=\check f$, with $f\in A(G)\cap\CF\CL_0(G)$, form a dense subspace of $L^2(G)$, it follows that the maps $\tilde\Lambda$ and~\eqref{eq:tildelambda2} are equal.
\ep

As was shown in~\cite{DC}, the unitary operator $\tilde J J$ must belong to $W^*(G)$.  The following makes this explicit.
\begin{lemma}
\label{UX}
For any $\vf\in C_c(G)$ and $g\in G$ we have:
$$
(\mathcal J \mathcal U\mathcal J \vf)(g)
=\int_{G} e^{i\langle q^\flat \xi_0-\xi_0,v\rangle}\,(\lambda_{(q,v)}\vf)(g)\, \Delta(q)^{-1/2}\,d(q,v).
$$
\end{lemma}
\begin{proof}
We have, with absolutely convergent integrals:
\begin{align*}
(\mathcal U\vf)(q,v)&=(U_\phi\CF_{V}^*\hat W_{\hat V}^*U_\phi^*\CF_V\vf)(q,v)\\
&=\int_{\hat V} e^{i\langle\xi,v\rangle}\,(\hat W_{\hat V}^*U_\phi^*\CF_V\vf)(q^\flat\xi_0,\xi)\,d\xi\\
&=\int_{\hat V} e^{i\langle\xi,v\rangle}\,(U_\phi^*\CF_V\vf)(q^\flat\xi_0-\xi,\xi)\,d\xi\\
&=\int_{\hat V\times V} e^{i\langle\xi,v-v'\rangle}\,\vf\big(\phi^{-1}(q^\flat\xi_0-\xi),v'\big)\,d\xi\, dv'\\
&=\int_{Q\times V} e^{-i\langle q'^\flat\xi_0-q^\flat\xi_0,v-v'\rangle}\,\vf(q',v')\,\frac{d_Q(q')\, dv'}{|q'|}\\
&=\int_{G} e^{i\langle (qq')^\flat\xi_0-q^\flat\xi_0,qv'\rangle}\,\vf\big((q,v)(q',v')\big)\,d(q',v')\\
&=\int_{G} e^{i\langle q'^\flat\xi_0-\xi_0,v'\rangle}\,\Delta(q')^{-1/2}\,(\rho_{(q',v')}\vf)(q,v)\,d(q',v').
\end{align*}
As $\mathcal J \rho_g\mathcal J=\lambda_g$, applying this to $\CJ\vf$ instead of $\vf$ we get the announced formula.
\end{proof}

By \cite[Proposition 5.4]{DC}, we deduce that the multiplicative unitary $\hat W_\Omega$ for the deformed quantum group $(W^*(G),\Omega\Dhat(\cdot)\Omega^*)$ is given by the formula
\begin{equation} \label{eq:mult-unitary}
\hat W_\Omega
=(\CJ \U J\CJ\otimes\hat J)\,\Omega\,\hat W^*\,(J\otimes\hat J)\,\Omega^*.
\end{equation}

Conjugating by the partial Fourier transform, we get a more explicit formula:

\begin{theorem}\label{thm:mult-unitary}
Let $G=Q\ltimes V$ be a second countable locally compact group satisfying the dual orbit Assumption~\ref{Fro}, and $\Omega$ be the dual unitary $2$-cocycle defined by~\eqref{2CC}. Then for the multiplicative unitary $\hat W_\Omega$ of the deformed quantum group $(W^*(G),\Omega\Dhat(\cdot)\Omega^*)$ and any $f\in L^2(G\times G)$ we have
\begin{multline*}
\big((\CF_V\otimes\CF_V)\hat W_\Omega (\CF_V^*\otimes\CF_V^*)f\big)(q_1,\xi_1;q_2,\xi_2)=|\phi^{-1}(\phi(q_2^{-1})+\xi_1)|\\
\hspace{1,5cm}\times f\big(q_2q_1,q_2^\flat\xi_1;\phi^{-1}(\phi(q_2^{-1})+\xi_1)^{-1}\phi^{-1}(\xi_0+\xi_1),{\phi^{-1}(\phi(q_2^{-1})+\xi_1)^{-1}}^\flat
({q_2^{-1}}^\flat\xi_2-\xi_1)\big).
\end{multline*}
\end{theorem}

\begin{proof}
We know by Lemma \ref{lem2} that for $f\in L^2(Q\times\hat V\times G, |q_1|^{-1}d_Q(q_1)\,d\xi_1\, dg)$, we have
$$
((\CF_V\otimes 1)\Omega (\CF_V^*\otimes 1)f)(q_1,\xi_1;q_2,v_2)
=
f\big(q_1,\xi_1;\phi^{-1}(\xi_0+\xi_1)q_2,\phi^{-1}(\xi_0+\xi_1)v_2\big).
$$
From this we obtain, for $f\in L^2(G\times G)$:
$$
\big((\CF_V\otimes \CF_V)\Omega (\CF_V^*\otimes \CF_V^*)f\big)(q_1,\xi_1;q_2,\xi_2)
= |\phi^{-1}(\xi_0+\xi_1)|^{-1}
f\big(q_1,\xi_1;\phi^{-1}(\xi_0+\xi_1)q_2,\phi^{-1}(\xi_0+\xi_1)^\flat \xi_2\big).
$$
It follows that
\begin{multline*}
\big((\CF_V\otimes \CF_V)\Omega ^*(\CF_V^*\otimes \CF_V^*)f\big)(q_1,\xi_1;q_2,\xi_2)\\
= |\phi^{-1}(\xi_0+\xi_1)|
f\big(q_1,\xi_1;\phi^{-1}(\xi_0+\xi_1)^{-1}q_2,{\phi^{-1}(\xi_0+\xi_1)^{-1}}^\flat \xi_2\big).
\end{multline*}

On the other hand, with the help of Lemma \ref{UX} we can perform calculations similar to those of Lemma \ref{lem2}
to get, for $f\in L^2(Q\times\hat V, |q|^{-1}d_q(q)\,d\xi)$, that
$$
(\CF_V\CJ\CU\CJ\CF_V^*f)(q,\xi)=\frac{|\phi^{-1}(\xi_0+\xi)|^{3/2}}{\Delta_Q\big(\phi^{-1}(\xi_0+\xi)\big)^{1/2}}
f\big(\phi^{-1}(\xi_0+\xi)^{-1}q,{\phi^{-1}(\xi_0+\xi)^{-1}}^\flat\xi\big).
$$
Moreover, one easily finds that
\begin{align*}
(\CF_V J\CF_V^*\,f)(q,\xi)=\overline{f(q,-\xi)},\qquad
(\CF_V \hat J\CF_V^*\,f)(q,\xi)=\frac{|q|^{3/2}}{\Delta_Q(q)^{1/2}}\overline{f\big(q^{-1},{q^{-1}}^\flat\xi\big)},
\end{align*}
and
\begin{equation}\label{eq:mult-unitary2}
\big((\CF_V\otimes \CF_V)\hat W^* (\CF_V^*\otimes \CF_V^*)f\big)(q_1,\xi_1;q_2,\xi_2)
= |q_2| f\big(q_2^{-1}q_1,{q_2^{-1}}^\flat\xi_1;q_2,\xi_1+\xi_2\big).
\end{equation}
Hence we get

\smallskip

$\displaystyle
\big((\CF_V\otimes\CF_V)\hat W_\Omega (\CF_V^*\otimes\CF_V^*)f\big)(q_1,\xi_1;q_2,\xi_2)$
\begin{align*}
&=
\big((\CF_V\otimes\CF_V)(\CJ \U \CJ\otimes1)(J\otimes\hat J)\,\Omega\,\hat W^*\,(J\otimes\hat J)\,\Omega^* (\CF_V^*\otimes\CF_V^*)f\big)(q_1,\xi_1;q_2,\xi_2)\\
&=\frac{|\phi^{-1}(\xi_0+\xi_1)|^{3/2}}{\Delta_Q\big(\phi^{-1}(\xi_0+\xi_1)\big)^{1/2}}
\big((\CF_V\otimes\CF_V)(J\otimes\hat J)\,\Omega\,\hat W^*\,(J\otimes\hat J)\,\Omega^* (\CF_V^*\otimes\CF_V^*)f\big)\\
&\qquad\qquad\qquad\qquad \qquad\qquad\big(\phi^{-1}(\xi_0+\xi_1)^{-1}q_1,{\phi^{-1}(\xi_0+\xi_1)^{-1}}^\flat\xi_1;q_2,\xi_2\big)\\
&=\frac{|\phi^{-1}(\xi_0+\xi_1)|^{3/2}|q_2|^{3/2}}{\Delta_Q\big(\phi^{-1}(\xi_0+\xi_1)\big)^{1/2}\Delta_Q(q_2)^{1/2}}
\big((\CF_V\otimes\CF_V)\Omega\,\hat W^*\,(J\otimes\hat J)\,\Omega^* (\CF_V^*\otimes\CF_V^*)\bar f\big)\\
&\qquad\qquad\qquad\qquad\qquad\qquad {\big(\phi^{-1}(\xi_0+\xi_1)^{-1}q_1,-{\phi^{-1}(\xi_0+\xi_1)^{-1}}^\flat\xi_1;q_2^{-1},{q_2^{-1}}^\flat\xi_2\big)}.
\end{align*}
Using that
$$
\phi^{-1}(\xi_0-{\phi^{-1}(\xi_0+\xi_1)^{-1}}^\flat\xi_1)=\phi^{-1}(\xi_0+\xi_1)^{-1},
$$
the expression above becomes
\begin{align*}
&\frac{|\phi^{-1}(\xi_0+\xi_1)|^{5/2}|q_2|^{3/2}}{\Delta_Q\big(\phi^{-1}(\xi_0+\xi_1)\big)^{1/2}\Delta_Q(q_2)^{1/2}}
\big((\CF_V\otimes\CF_V)\hat W^*\,(J\otimes\hat J)\,\Omega^* (\CF_V^*\otimes\CF_V^*)\bar f)\\
&\qquad\qquad\big(\phi^{-1}(\xi_0+\xi_1)^{-1}q_1,-{\phi^{-1}(\xi_0+\xi_1)^{-1}}^\flat\xi_1;\phi^{-1}(\xi_0+\xi_1)^{-1}q_2^{-1},
{\phi^{-1}(\xi_0+\xi_1)^{-1}}^\flat{q_2^{-1}}^\flat\xi_2\big)\\
&=\frac{|\phi^{-1}(\xi_0+\xi_1)|^{3/2}|q_2|^{1/2}}{\Delta_Q\big(\phi^{-1}(\xi_0+\xi_1)\big)^{1/2}\Delta_Q(q_2)^{1/2}}
\big((\CF_V\otimes\CF_V)(J\otimes\hat J)\,\Omega^* (\CF_V^*\otimes\CF_V^*)\bar f\big)\\
&\qquad\qquad\big(q_2q_1,-q_2^\flat\xi_1;\phi^{-1}(\xi_0+\xi_1)^{-1}q_2^{-1},
{\phi^{-1}(\xi_0+\xi_1)^{-1}}^\flat({q_2^{-1}}^\flat\xi_2-\xi_1)\big)\\
&=\frac1{|q_2|}
\big((\CF_V\otimes\CF_V)\Omega^* (\CF_V^*\otimes\CF_V^*)f\big)\big(q_2q_1,q_2^\flat\xi_1;q_2\phi^{-1}(\xi_0+\xi_1),
\xi_2-q_2^\flat\xi_1\big)\\
&=|\phi^{-1}(\phi(q_2^{-1})+\xi_1)|\\
&\qquad\qquad f\big(q_2q_1,q_2^\flat\xi_1;\phi^{-1}(\phi(q_2^{-1})+\xi_1)^{-1}\phi^{-1}(\xi_0+\xi_1),{\phi^{-1}(\phi(q_2^{-1})+\xi_1)^{-1}}^\flat
({q_2^{-1}}^\flat\xi_2-\xi_1)\big),
\end{align*}
which is what we need.
\end{proof}

Recall that in Section~\ref{ss:deform} we considered a continuous family of cohomologous dual unitary $2$-cocycles $\Omega_q$, $q\in Q$, defined by replacing $\xi_0$ by $\xi_0^q=q^\flat\xi_0$.

\begin{corollary}
We have:
\begin{enumerate}
\item[(i)] the map $Q\ni q\mapsto \hat W_{\Omega_q}$ is $so$-continuous;
\item[(ii)] if $\phi\colon Q\to\hat V$ is open and $z_n^\flat\to0$ in $\End(\hat V)$ pointwise for a sequence of elements $z_n\in Z(Q)$, then $\hat W_{\Omega_{z_n^{-1}}}\to\hat W$ in the $so$-topology.
\end{enumerate}
\end{corollary}

\bp
Part (i) follows already from formula~\eqref{eq:mult-unitary}. Indeed, the map $q\mapsto\Omega_q$ is continuous by Proposition~\ref{prop:cont}. On the other hand, the unitary $\CU_q$ in formula~\eqref{eq:mult-unitary} for the dual cocycle $\Omega_q$ is given by Lemma~\ref{lem:U}, with $\xi_0$ replaced by $\xi^q_0$. To be more precise, that lemma is formulated under the assumption that the Haar measure on $Q$ is normalized by~\eqref{eq:dqdxi}. If we replace $\xi_0$ by $\xi^q_0=q^\flat\xi_0$, and, correspondingly, the map $\phi$ by $\phi_q(q')=\phi(q'q)$, but want to keep the same measure on $Q$, then the map
$$
 U_{\phi_q} :L^2(\hat V\times V)\to L^2(G),\quad
U_{\phi_q} f(q',v'):=f\big(\phi_q(q'),v' \big),
$$
is unitary only up to a scalar factor. But this means that for Lemma~\ref{lem:U} to remain true we just have to state it as the equality
$$
\U_q=U_{\phi_q}\CF_{V}^*\hat W_{\hat V}^*\CF_VU_{\phi_q}^{-1}.
$$
As the map $q\mapsto U_{\phi_q}$ is obviously continuous in the $so$-topology, we conclude that the map $q\mapsto\CU_q$ is continuous as well, hence so is the map $q\mapsto\hat W_{\Omega_q}$.

\smallskip

In order to prove (ii), recall that in the proof of Proposition~\ref{prop:1-deform} we already showed that if~$\phi$ is open, then $\phi^{-1}(\xi_0+z_n^\flat\xi)\to e$ as $z_n^\flat\to0$. It follows that, for all $q\in Q$,
$$
\phi^{-1}_{z_n^{-1}}({q^{-1}}^\flat\xi^{z_n^{-1}}_0+\xi)=q^{-1}\phi^{-1}(\xi_0+z_n^\flat\xi)\to q^{-1}.
$$
By Theorem~\ref{thm:mult-unitary} we then conclude that $(\CF_V\otimes\CF_V)\hat W_{\Omega_{z_n^{-1}}} (\CF_V^*\otimes\CF_V^*)\to Y$, where the operator~$Y$ is given by
$$
(Yf)(q_1,\xi_1;q_2,\xi_2)=|q_2|^{-1} f(q_2q_1,q_2^\flat\xi_1;q_2,\xi_2-q_2^\flat\xi_1).
$$
Since by~\eqref{eq:mult-unitary2} we have $Y=(\CF_V\otimes\CF_V)\hat W (\CF_V^*\otimes\CF_V^*)$, this proves the result.
\ep

\subsection{Stachura's dual cocycle}

In this section we consider the simplest example of our setup, the $ax+b$ group $G$ over the reals. In this case Stachura~\cite{Stachura} already defined a dual cocycle on~$G$. We refer the reader to his paper for a motivation of the construction and just present an explicit form of the cocycle.

Consider the following operators affiliated with $W^*(G)$:
$$
X:=i\frac d{dt}\lambda_{(e^{t},0)}\Big|_{t=0},\ \ Y:=i\frac d{dt}\lambda_{(1,t)}\Big|_{t=0},\ \ I:=\lambda_{(-1,0)}.
$$
Then the dual unitary cocycle on  $G$  found by Stachura, see \cite[Lemma 5.6]{Stachura}, is defined by
\begin{equation}\label{eq:Stachura}
\Omega_S:=\exp\big\{iX\otimes\log|1+Y|\big\}\,{\rm Ch}\big(1\otimes {\rm sgn}(1+Y), I\otimes 1\big),
\end{equation}
where ${\rm Ch}:\{-1,1\}\times\{-1,1\}\to \{-1,1\}$ is the unique nontrivial bicharacter.\footnote{In fact, Stachura works in a representation of $G$ equivalent to the regular representation. His operators $X,Y,I$ are the operators on $L^2(\R^*\times\R,q^{-2}dq\,d\xi)$ given by
$$
(X f)(q,\xi)=i(q\partial_q+\xi\partial_\xi)f(q,\xi),\quad (Y f)(q,\xi)=\xi f(q,\xi),\quad (I f)(q,\xi)= f(-q,-\xi).
$$
The equivalence is implemented by the unitary $(2\pi)^{-1/2}\CF_\R U\colon L^2(G)\to L^2(\R^*\times\R,q^{-2}dq\,d\xi)$, where $U\colon L^2(G)\to L^2(G)$ is defined by $(Uf)(q,v)= |q|f(q^{-1},v)$.
}

\begin{proposition} \label{prop:Stachura}
The dual cocycle $\Omega_S$ coincides with the dual cocycle $\overline{\Omega}$ defined by~\eqref{eq:anti-Omega} for $\xi_0=-1$. In particular, $\Omega_S$ is cohomologous to the dual cocycle $\Omega$ defined by~\eqref{2CC}.
\end{proposition}

\bp
Observe that
$$
 2{\rm Ch}(\eps_1,\eps_2)=1+\eps_1+\eps_2-\eps_1\eps_2.
$$
Hence, using that $ I\otimes 1$ commutes with $ X\otimes 1$ and $1\otimes Y$,  we get
\begin{multline*}
2\Omega_{S}=\exp\big\{i X\otimes\log|1+Y|\big\}\big(1+1\otimes {\rm sgn}(1+ Y)\big) \\
+\big(I\otimes 1\big) \exp\big\{iX\otimes\log|1+Y|\big\}\big(1-1\otimes {\rm sgn}(1+ Y)\big).
\end{multline*}
In terms of the functions $F_\eps:\R^2\to \T$, $\eps\in\{0,1\}$, defined by
$$
F_\eps(x,y):=\exp\{ix\ln|1+y|\}\,{\rm sgn}(1+y)^\eps,
$$
we therefore have
\begin{multline*}
 2\Omega_{S}=\big(F_0(X\otimes1,1\otimes Y)+F_1(X\otimes1,1\otimes Y)\big)\\
 + \big(I\otimes 1\big)\big(F_0(X\otimes1,1\otimes Y)-F_1(X\otimes1,1\otimes Y)\big).
\end{multline*}
Normalizing the $2$-dimensional Fourier transform by
$$
(\mathcal F_{\R^2} f)(s,t)=\frac1{2\pi}\int_{\R^2} e^{-i(xs+yt)}\;f(x,y)\;dx\,dy,
$$
we then obtain
\begin{align*}
\Omega_{S}&=\frac1{4\pi}\int_{\R^2}\big((\mathcal F_{\R^2}^*  F_0)(s,t)+(\mathcal F_{\R^2}^*  F_1)(s,t)\big)\;e^{-isX}\otimes e^{-itY}\,ds\,dt  \\
&\qquad\qquad+ \frac1{4\pi}\int_{\R^2}\big((\mathcal F_{\R^2}^*  F_0)(s,t)-(\mathcal F_{\R^2}^* F_1)(s,t)\big)\;Ie^{-isX}\otimes e^{-itY}\,ds\,dt.\\
&=\frac1{4\pi}\int_{\R^2}\big((\mathcal F_{\R^2}^*  F_0)(s,t)+(\mathcal F_{\R^2}^*  F_1)(s,t)\big)\;\lambda_{(e^{s},0)}\otimes \lambda_{(1,t)}\,ds\,dt  \\
&\qquad\qquad+ \frac1{4\pi}\int_{\R^2}\big((\mathcal F_{\R^2}^*  F_0)(s,t)-(\mathcal F_{\R^2}^* F_1)(s,t)\big)\;\lambda_{(-e^{s},0)}\otimes \lambda_{(1,t)}\,ds\,dt,
\end{align*}
with the integrals understood in the distributional sense.

We now need to compute the inverse Fourier transforms (in the sense of tempered distributions):
\begin{align*}
(\mathcal F_{\R^2}^*  F_\eps)(s,t)=\frac1{2\pi}\int_{\R^2} e^{i(xs+yt)}\;\exp\{ix\ln|1+y|\}\,{\rm sgn}(1+y)^\eps\;dx\,dy.
\end{align*}
Note first that
$$
\frac1{2\pi}\int_{\R} e^{ix(s+\ln|1+y|)}\;dx=\delta_0\big(s+\ln|1+y|\big).
$$
Consider now, for fixed $s\in\R$, the function $\varphi(y):= s+\ln|1+y|$. It  has two simple zeros located at $y_\pm=\pm e^{-s}-1$
and  it is continuously differentiable (away from $-1$) with $\varphi'(y_\pm)=\pm e^{s}$. Therefore
$$
\delta_0\big(s+\ln|1+y|\big)=|\varphi'(y_+)|^{-1}\delta_{y_+}(y)+|\varphi'(y_-)|^{-1}\delta_{y_-}(y)=e^{-s}\big(\delta_{-1+e^{-s}}(y)+\delta_{-1-e^{-s}}(y)\big).
$$
Hence we get
\begin{align*}
(\CF_{\R^2}^* F_\eps)(s,t)
&=e^{-s}\int_{\R} e^{i yt}\;\big(\delta_{-1+e^{-s}}(y)+\delta_{-1-e^{-s}}(y)\big)\,{\rm sgn}(1+y)^\eps\;dy\\
&=e^{-s-it}\big(e^{ie^{-s}t}+(-1)^\eps e^{-i e^{-s}t}\big).
\end{align*}
From this we obtain
$$
(\mathcal F_{\R^2}^* F_0)(s,t)+(\mathcal F_{\R^2}^* F_1)(s,t)=2e^{-s-it}e^{ie^{-s}t}\quad\mbox{and}\quad
(\mathcal F_{\R^2}^* F_0) (s,t)- (\mathcal F_{\R^2}^* F_1)(s,t)=2e^{-s-it}e^{-ie^{-s}t},
$$
and therefore
\begin{align*}
\Omega_{S}&=\frac1{2\pi}\int_{\R\times\R}e^{-s-it}\big(e^{ie^{-s}t}\;\lambda_{(e^{s},0)}+e^{-ie^{-s}t}\;\lambda_{(-e^{s},0)}\big)\otimes \lambda_{(1,t)}\,ds\,dt.
\end{align*}
Setting now $q=e^{s}$ and $v=-t$ we get
\begin{align*}
\Omega_{S}&=\frac1{2\pi}\int_{\R_+\times\R}q^{-1}e^{iv}\big(e^{-i{q^{-1}} v}\;\lambda_{(q^{-1},0)}+e^{i{q^{-1}} v}\;\lambda_{(-q^{-1},0)}\big)\otimes \lambda_{(1,-v)}\,q^{-1}dq\,dv\\
 &=\frac1{2\pi}\int_{\R^*\times\R}e^{i(v-{q^{-1}}v)}\;\lambda_{(q^{-1},0)}\otimes \lambda_{(1,-v)}\,\frac{dq\,dv}{q^2}.
\end{align*}
Setting $\xi_0=-1$ and remembering that we have here  $q^\flat\xi=q^{-1}\xi$ and $d(q,v)=(2\pi)^{-1}q^{-2}dq\,dv$, we finally get
$$
\Omega_{S}=\int_Ge^{i\langle q^\flat \xi_0-\xi_0,v\rangle}\,\lambda_{(q,0)^{-1}}\otimes\lambda_{(1,v)^{-1}} \,d(q,v),
$$
which is exactly the dual cocycle $\overline{\Omega}$ defined by~\eqref{eq:anti-Omega} in Remark \ref{BO}. As we already observed there, $\overline{\Omega}$ is
cohomologous to $\Omega$.
\ep

As was suggested by Stachura~\cite{Stachura}, the quantum group $(W^*(G),\Omega_S\Dhat(\cdot)\Omega_S^*)$ is isomorphic to the quantum $ax+b$ group of Baaj and Skandalis~\cite{BS2} (see also~\cite[Section~5.3]{VV}), but his arguments fall a bit short of proving that this is indeed the case. The above proposition together with Theorem~\ref{thm:bicrossed} below complete his work.

\section{Bicrossed product construction}\label{sec:bicross}

Recall that a pair $(G_1,G_2)$ of closed subgroups of a locally compact second countable group~$G$ is called a \emph{matched pair} if $G_1\cap G_2=\{e\}$ and $G_1G_2$ is a subset of $G$ of full measure~\cite{BSV}. Given such a pair, we have almost everywhere defined measurable left actions $\alpha$ of $G_1$ and $\beta$ of $G_2$ on the measure spaces $G_2$ and $G_1$, resp., such that
$$
gs^{-1}=\alpha_g(s)^{-1}\beta_s(g)\ \ \text{for}\ \ g\in G_1,\ s\in G_2.
$$
We can then define a bicrossed product $\hat G_1\bicrossl G_2$. This is a locally compact quantum group with the function algebra
$$
L^\infty(\hat G_1\bicrossl G_2):=G_1\ltimes_\alpha L^\infty(G_2).
$$
The coproduct on $L^\infty(\hat G_1\bicrossl G_2)$ is a bit more difficult to describe, but we will not need to know the exact definition and refer the reader for that to~\cite{BSV} or~\cite{VV}.\footnote{To be more precise, we are considering the quantum group $(M,\Delta)$ from~\cite[Section~4.2]{VV}, with $i\colon G_1\to G$ and $j\colon G_2\to G$ defined by $i(g)=g$, $j(s)=s^{-1}$. This is the same as the quantum group $(\hat M,\hat\Delta)$ from~\cite[Section~3]{BSV}, see the discussion following~\cite[Definition~3.3]{BSV}.} Then by~\cite[Propostion~2.9 and Theorem~2.13]{VV} the dual quantum group is $G_1\bicrossr \hat G_2=\hat G_2\bicrossl G_1$, the bicrossed product defined by the matched pair $(G_2,G_1)$ of subgroups of $G$.

\begin{theorem}\label{thm:bicrossed}
Let $G=Q\ltimes V$ be a second countable locally compact group satisfying the dual orbit Assumption~\ref{Fro} and $\Omega$ be the dual unitary $2$-cocycle $\Omega$ on $G$
defined by~\eqref{2CC}. Then the quantum group $(W^*(G),\Omega\Dhat(\cdot)\Omega^*)$ is isomorphic to the bicrossed product quantum group defined by the matched pair $(Q,\xi_0 Q\xi_0^{-1})$ of subgroups of $Q\ltimes\hat V$.
\end{theorem}

Here by $\xi_0^{-1}$ we of course mean the element $(\id,-\xi_0)\in\Aff(\hat V)$.

\bp
Consider the measure space $X=Q\times\hat V$, with the measure class defined by the product of Haar measures.
By Theorem~\ref{thm:mult-unitary} the multiplicative unitary $\hat W_\Omega$ is unitarily conjugate to the unitary associated with the measurable (almost everywhere defined)
transformation $v\colon X\times X\to X\times X$ given by
\begin{multline*}
v(q_1,\xi_1;q_2,\xi_2)=\big(q_2q_1,q_2^\flat\xi_1;\\
\phi^{-1}({q_2^{-1}}^\flat\xi_0+\xi_1)^{-1}\phi^{-1}(\xi_0+\xi_1),{\phi^{-1}({q_2^{-1}}^\flat\xi_0+\xi_1)^{-1}}^\flat
({q_2^{-1}}^\flat\xi_2-\xi_1)\big).
\end{multline*}
Hence this transformation is \emph{pentagonal}~\cite{BS3}. By a result of Baaj and Skandalis~\cite{BS3}, see also~\cite[Proposition~5.1]{BSV} for a correction, under mild technical assumptions the pentagonal transformations arise from matched pairs of groups. Let us follow the proof in~\cite{BS3} and see which pair we get.

Following~\cite{BS3} we write the transformation $v$ as $v(x,y)=(x\bullet y,x\# y)$. The maps $(x,y)\mapsto (x\bullet y,y)$ and $(x,y)\mapsto(x,x\# y)$ are measure class isomorphisms, so the assumptions of \cite[Proposition~5.1]{BSV} are satisfied and therefore the pentagonal transformation $v$ and the multiplicative unitary $\hat W_\Omega$ indeed come from a matched pair of groups.

It is not difficult to check that the inverse map $v^{-1}$, which we will write as $v^{-1}(x,y)=(x\diamond y,x*y)$, is given by
\begin{align*}
v^{-1}(q_1,\xi_1;q_2,\xi_2)&=\big(\phi^{-1}(\phi^{-1}(\xi_0+\xi_1)^\flat q_2^\flat\xi_0-\xi_1)^{-1}q_1,{\phi^{-1}(\phi^{-1}(\xi_0+\xi_1)^\flat q_2^\flat\xi_0-\xi_1)^{-1}}^\flat\xi_1;\\
&\qquad\qquad\phi^{-1}(\phi^{-1}(\xi_0+\xi_1)^\flat q_2^\flat\xi_0-\xi_1),\xi_1+\phi^{-1}(\xi_0+\xi_1)^\flat\xi_2\big).
\end{align*}

By \cite[Lemma 2.1]{BS3}, there exists a second countable locally compact group $G_1$, a right action of $G_1$ on $X$ and  an equivariant measurable map $f_1\colon X\to G_1$ such that for almost all pairs $(x,y)\in X\times X$ we have
$$
x\bullet y=xf_1(y).
$$
Although this is not explicitly stated in~\cite{BS3}, it is not difficult to see that the group $G_1$, the action of $G_1$ on $X$ and the map $f_1$ are uniquely determined by these properties up to an isomorphism. In our case it is easy to see what we get:
$$
G_1=Q,\ \ (q,\xi)q_1=(q_1^{-1}q,{q_1^{-1}}^\flat\xi), \ \ f_1(q,\xi)=q^{-1}.
$$

In a similar way there exist a (unique up to isomorphism) second countable locally compact group $G_2$, a right action of $G_2$ on $X$ and  an equivariant measurable map $f_2\colon X\to G_2$ such that for almost all pairs $(x,y)\in X\times X$ we have
$$
x* y=yf_2(x).
$$
In our case we get
$$
G_2=Q,\ \ f_2(q,\xi)=\phi^{-1}(\xi+\xi_0)^{-1},\ \ (q,\xi)q_2=(\phi^{-1}({q_2^{-1}}^\flat(q^\flat\xi_0-\xi_0)+\xi_0),{q_2^{-1}}^\flat(\xi+\xi_0)-\xi_0).
$$

According to \cite{BS3,BSV} we then get a locally compact group $G'$ and embeddings $h_i\colon G_i\to G$ of the groups $G_i$ as closed subgroups such that the map $G_1\times G_2\to G'$, $(s,g)\mapsto h_1(s)h_2(g)$ is injective and the complement of its image is a set of measure zero. By \cite[Lemma~3.5(b)]{BS3}, for almost all $(x,y)\in X\times X$ we have
$$
h_2(f_2(x))h_1(f_1(y))=h_1(f_1(b))h_2(f_2(a)),\ \ \text{where}\ \ (a,b)=v(x,y).
$$
Again, it is not difficult to see that these properties completely determine the locally compact group $G'$ up to isomorphism. In our case the above identity reads
$$
h_2(\phi^{-1}(\xi+\xi_0)^{-1})h_1(q^{-1})=h_1(\phi^{-1}(\xi+\xi_0)^{-1}\phi^{-1}({q^{-1}}^\flat\xi_0+\xi))h_2(\phi^{-1}(q^\flat\xi+\xi_0)^{-1}).
$$
Letting $q_1=q$ and $q_2=\phi^{-1}(\xi+\xi_0)$ we equivalently get
$$
h_2(q_2^{-1})h_1(q^{-1})=h_1(q_2^{-1}\phi^{-1}({q_1^{-1}}^\flat\xi_0+q_2^\flat\xi_0-\xi_0))h_2(\phi^{-1}(q_1^\flat(q_2^\flat\xi_0-\xi_0)+\xi_0)^{-1}),
$$
or in other words,
\begin{equation*}
h_1(q_1)h_2(q_2)=h_2(\phi^{-1}(q_1^\flat q_2^\flat\xi_0-q_1^\flat\xi_0+\xi_0))h_1(\phi^{-1}(q_1^\flat q_2^\flat\xi_0-q_1^\flat\xi_0+\xi_0)^{-1}q_1q_2).
\end{equation*}
We then see that these properties are satisfied by the group $G':=Q\ltimes\hat V$ and the embeddings
$$
h_1(q):=q=(q,0),\ \ h_2(q):=\xi_0q\xi_0^{-1}=(q,\xi_0-q^\flat\xi_0).
$$
Therefore we conclude that $\hat W_\Omega$ is unitarily conjugate to the unitary $W$ from~\cite{BSV} defined by the matched pair $(Q,\xi_0Q\xi_0^{-1})$ of subgroups of $Q\ltimes\hat V$, or equivalently, to the unitary $\hat W$ from~\cite{VV} defined by the matched pair $(\xi_0Q\xi_0^{-1},Q)$, see the discussion following~\cite[Definition~3.3]{BSV}. Therefore
$(W^*(G),\Omega\Dhat(\cdot)\Omega^*)$ is isomorphic to the dual of the bicrossed product defined by $(\xi_0Q\xi_0^{-1},Q)$, hence to the bicrossed product defined by $(Q,\xi_0Q\xi_0^{-1})$.
\ep

\begin{corollary} \label{cor:bicrossed}
The locally compact quantum group $(W^*(G),\Omega\Dhat(\cdot)\Omega^*)$ is self-dual. If $G$ is nontrivial, then this quantum group is noncompact and nondiscrete, and if $G$ is nonunimodular (that is, $\Delta_Q\ne|\cdot|_V$), then the quantum group is also nonunimodular, with nontrivial scaling group and scaling constant $1$.
\end{corollary}

\bp
In order to prove self-duality it suffices to show that the matched pairs $(Q,\xi_0Q\xi_0^{-1})$ and $(\xi_0Q\xi_0^{-1},Q)$ are isomorphic. The conjugation by the element $(-\id,\xi_0)\in\Aff(\hat V)$ gives such an isomorphism. Next, the quantum group $(W^*(G),\Omega\Dhat(\cdot)\Omega^*)$ cannot be discrete, since~$W^*(G)$ is a factor. By self-duality it then cannot be compact either. The rest follows by \cite[Propositions~4.15,~4.16]{VV}.
\ep

\bigskip


\begin{thebibliography}{99}

\bibitem{A}
P. Aniello, ``Square integrable projective representations and square integrable representations modulo a relatively central subgroup'',
Int. J. Geom. Methods Mod. Phys. {\bf 3} (2006), no.~2, 233--267.


\bibitem{BS3}
S. Baaj and G. Skandalis, ``Transformations pentagonales''. C.R. Acad. Sci., Paris, S\'er. I {\bf 327} (1998), 623--628.

\bibitem{BSV} S. Baaj, G. Skandalis and S. Vaes, ``Non-semi-regular quantum groups coming from number theory'',
Comm. Math. Phys.  {\bf235} (2003), 139--167.

\bibitem{BG}
P. Bieliavsky and V. Gayral, ``Deformation quantization for actions of K\"ahlerian Lie groups'', Mem. Amer. Math. Soc. {\bf 236} (2015), no.~1115;
available with erratum at  \href{https://arxiv.org/abs/1109.3419}{arXiv:1109.3419v7 [math. OA]}.

\bibitem{BGNT1}
P. Bieliavsky, V. Gayral, S. Neshveyev and L. Tuset, ``On deformation of C$^*$-algebras by actions of K\"ahlerian Lie groups", Int. J. Math. {\bf 27} (2016), 1650023,
and Internat. J. Math. 30 (2019), no. 11, 1992002.

\bibitem{BP}
A. Borowiec and A. Pachol, ``kappa-Minkowski spacetime as the result of Jordanian twist deformation'', Phys. Rev. D {\bf 79}, 045012 (2009).

\bibitem{Bruhat}
F. Bruhat, ``Distributions sur un groupe localement compact et applications \`a l'\'etude des repr\'esentations des groupes $p$-adiques'', Bull. Soc. Math. France {\bf89} (1961), 43--75.

 \bibitem{CGG}
 V. Coll, M. Gerstenhaber and A. Giaquinto, ``An explicit deformation formula with non-commuting derivations'', Israel Math. Conf. Proc. Ring Theory, vol. 1, Weizmann Science Press, New York, 1989, pp. 396--403.

\bibitem{CT}
A. Connes and M. Takesaki, ``The flow of weights on factors of type III'', T\^{o}hoku Math. J. (2) {\bf 29} (1977), no.~4, 473--575.

\bibitem{Da}
A.A. Davydov, ``Galois algebras and monoidal functors between categories of representations of finite groups'', J. Algebra {\bf 244} (2001), no.~1, 273--301.

\bibitem{DC}
K. De Commer, ``Galois objects and cocycle twisting for locally compact quantum groups'',
J. Operator Theory {\bf66} (2011), no$.$ 1, 59 -- 106.

\bibitem{DC2}
 K. De Commer, ``I-factorial quantum torsors and Heisenberg algebras of quantized universal enveloping type'', preprint arXiv:1702.08191v1 [math.OA].

\bibitem{Dr}
V.G. Drinfeld, ``Constant quasiclassical solutions of the Yang-Baxter quantum equation'' (Russian), Dokl. Akad. Nauk SSSR {\bf 273} (1983), no.~3, 531--535.

\bibitem{DM}
M. Duflo and C.C. Moore, ``On the regular representation of a nonunimodular locally compact group'',
J. Funct. Anal. {\bf 21} (1976), no.~2, 209--243.

\bibitem{EK}
P. Etingof and D. Kazhdan, ``Quantization of Lie bialgebras. I'', Selecta Math. (N.S.) {\bf 2} (1996), no.~1, 1--41.


\bibitem{GV}
H. Gr\" abe and A. Vlassov,
``On a formula of Coll-Gerstenhaber-Giaquinto'',
J. Geom. Phys. {\bf28} (1998), no. 1-2, 129--142.

\bibitem{J}
D. Jondreville, ``A locally compact quantum group arising from
quantization of the affine group of a local field'', Lett. Math. Phys. {\bf109} (2019), 781--797.

\bibitem{KK}
E. Koelink and J. Kustermans, ``A locally compact quantum group analogue of the normalizer of $\mathrm SU(1,1)$ in ${\mathrm SL}(2,\C)$'', Comm. Math. Phys. {\bf 233} (2003), no.~2, 231--296.

\bibitem{Kor}
L.I. Korogodsky, ``Quantum group ${\mathrm SU}(1,1)\rtimes Z_2$ and ``super-tensor'' products'', Comm. Math. Phys. {\bf 163} (1994), no.~3, 433--460.

\bibitem{KLM}
P.P. Kulish, V.D. Lyakhovsky and A.I. Mudrov, ``Extended Jordanian twists for Lie algebras'', J. Math. Phys. {\bf 40} (1999), no.~9, 4569--4586.

\bibitem{KLS}
P.P. Kulish, V.D. Lyakhovsky and A. Stolin, ``Chains of Frobenius subalgebras of $\mathfrak{so}(M)$ and the corresponding twists'', J. Math. Phys. {\bf 42} (2001), no.~10, 5006--5019.

\bibitem{MR}
S. Majid and H. Ruegg, ``Bicrossproduct structure of $\kappa$-Poincar\'e group and non-commutative geometry'', Phys. Lett. B {\bf334} (1994), no. 3--4, 348--354.

\bibitem{Moore}
C.C. Moore, ``Group extensions and cohomology for locally compact groups. III'',
Trans. Amer. Math. Soc. {\bf 221} (1976), no.~1, 1--33.

\bibitem{Mov}
M.V. Movshev, ``Twisting in group algebras of finite groups'', Funct. Anal. Appl. {\bf 27} (1993), no.~4, 240--244 (1994).

\bibitem{NTcoc}
S. Neshveyev and L. Tuset, ``On second cohomology of duals of compact groups'', Internat. J. Math. {\bf 22} (2011), no.~9, 1231--1260.

\bibitem{NT}
S. Neshveyev and L. Tuset, ``Deformation of C$^*$-algebras by cocycles on locally compact quantum groups'', Adv. Math. \textbf{254} (2014), 454--496.

\bibitem{Og}
O. Ogievetsky, ``Hopf structures on the Borel subalgebra of $\mathfrak{sl}(2)$'', in: Proceedings of Winter School in Geometry and Physics, Zdikov, January 1993, Supplemento ai Rendiconti del Circolo Matematico di Palermo, Serie II, {\bf 37} (1994), 185--199.

\bibitem{Oo}
A.I. Ooms, ``On Frobenius Lie algebras'', Comm. Algebra {\bf 8} (1980), no.~1, 13--52.

\bibitem{BS2}
G. Skandalis, ``Duality for locally compact Quantum groups'' (joint work with Baaj). Mathematisches
Forschungsinstitut Oberwolfach, Tagungsbericht 46/1991, $C^*$-algebren, (1991), p. 20.

\bibitem{Stachura}
P. Stachura, ``On the quantum `ax+b' group'', Journal of Geometry and Physics {\bf 73} (2013), 125--149.

\bibitem{Segal}
I. Segal, ``Transforms for operators and symplectic automorphisms over a locally compact Abelian group'', Math. Scand. {\bf 13} (1963), 31--43.


\bibitem{Stratila}
S. Str\v{a}til\v{a}, \emph{Modular theory in operator algebras}. Abacus Press, Tunbridge Wells, (1981).

\bibitem{VV}
S. Vaes and L. Vainerman, ``Extensions of locally compact quantum groups and the bicrossed product construction'', Adv. Math. {\bf 175}  (2003),  no. 1, 1--101.

\bibitem{Wa}
A. Wassermann, ``Ergodic actions of compact groups on operator algebras. II. Classification of full multiplicity ergodic actions'',
Canad. J. Math. {\bf 40} (1988), no.~6, 1482--1527.

\bibitem{Weil}
A. Weil, ``Sur certain groupes d'op\'erateurs unitaires'', Acta Math. {\bf111} (1964), 143--211.

\bibitem{Wor}
S. L. Woronowicz, ``Unbounded elements affiliated with $C^*$-algebras and noncompact quantum groups'', Comm. Math. Phys. {\bf 136} (1991), no.~2, 399--432.

\end{thebibliography}
\end{document}